\documentclass[12pt]{amsart}

\usepackage[british]{babel}  
\usepackage[top=1.5cm]{geometry}    
\usepackage{mathrsfs}
\usepackage{bbm}					
\usepackage[parfill]{parskip} 
 \usepackage{graphicx}				
\usepackage{breqn}
\usepackage{bm}	
\usepackage{amsmath,amssymb,amsthm,eucal,verbatim}
\usepackage{hyperref,amsbsy}
\usepackage{graphicx}
\usepackage{epstopdf}

\usepackage{xcolor}

\usepackage{enumitem}

\usepackage{mathtools}

\allowdisplaybreaks

\advance \topmargin by \headheight
\advance \topmargin by-\headsep
\evensidemargin \oddsidemargin
\theoremstyle{plain}
\newtheorem{thm}{Theorem}[section]

\newtheorem{cor}[thm]{Corollary}
\newtheorem{lem}{Lemma}[subsection]
\newtheorem*{thm*}{Theorem}
\newtheorem*{lem*}{Lemma}
\newtheorem*{cor*}{Corollary}

\newtheorem*{rem*}{Remark}
\newtheorem{prop}[thm]{Proposition}

\theoremstyle{definition}

\newcommand{\be}{\begin{equation}}
\newcommand{\ee}{\end{equation}}

\newcommand{\ve}{\varepsilon}

\makeatletter
\newcommand*{\bigss}[1]{{\hbox{$\left#1\vbox to11\p@{}\right.\n@space$}}}
\makeatother

\newcommand{\sdfrac}[2]{\mbox{\small$\displaystyle\frac{#1}{#2}$}}

\title[Rankin-Selberg $L$-functions in the prime-power level aspect]{Twisted Moments of Rankin-Selberg $L$-functions in the Prime-Power Level Aspect}

\thanks{This work was partially funded by the PIMS Collaborative Research Group {\it $L$-functions in Analytic Number Theory}. Fatma \c{C}i\c{c}ek was supported by a Pacific Institute for the Mathematical Sciences (PIMS) postdoctoral fellowship at the University of Northern British Columbia. The research of Alia Hamieh is supported by an NSERC discovery grant.}



\author[F. \c{C}\.{\i}\c{c}ek and A. Hamieh]{Fatma \c{C}\.{\i}\c{c}ek and Alia Hamieh}
\address{\.{I}stanbul, 34528, T\tiny{\"{U}}RK\.{I}YE}
\email{cicek.ftm@gmail.com}
\address{University of Northern British Columbia\\ Department of Mathematics and Statistics \\ 3333 University Way\\ Prince George, BC\ V2N 4Z9\\ Canada}
\email{alia.hamieh@unbc.ca}

\begin{document}

\begin{abstract}
We compute the twisted first and second moments of the shifted central values of the Rankin-Selberg $L$-functions given by $L\left(\frac12+\omega, f\otimes g\right)$ as $f$ varies over primitive forms of prime power level $p^\nu$ with $\nu \geq 3$. Here $\omega$ is a bounded shift and $g$ is a fixed primitive form of level relatively prime to $p$. 

\end{abstract}

\maketitle


\section{Introduction}

Let $S_k(N)$ be the space of cusp forms of even weight $k\geq 2$ with respect to the congruence subgroup $\Gamma_0(N)$. Any $f\in S_k(N)$ admits a Fourier expansion at infinity
\begin{equation*}
f(z)=\sum_{n\ge 1} \lambda_f(n)n^{\frac{k-1}{2}}\,e(nz),\quad e(z):=e^{2\pi i z}.
\end{equation*}
The space $S_k(N)$ is equipped with the Petersson inner product
\begin{equation*}
\langle f,g\rangle = \int_{F_0(N)} f(z)\overline{g(z)}\,y^{k}\,\frac{dx\,dy}{y^{2}},
\end{equation*}
where $F_0(N)$ is a fundamental domain for the action of $\Gamma_0(N)$ on the upper half-plane $\mathbb{H}$. The Hecke operators $T_n$, with $(n,N)=1$, are normal with respect to the inner product. One can thus find an orthogonal basis of $S_k(N)$, we denote it by $H_k(N)$, formed of eigenvectors of all the $\{T_n,(n,N)=1\}.$ For $f$ a Hecke eigenvector of $T_n$, $\psi_f(n)n^{(k-1)/2}$
denotes its eigenvalue.

According to Atkin-Lehner theory \cite{Atkin-Lehner}, the space $S_k(N)$ can be decomposed
into two subspaces
\begin{equation*}
S_k(N)=S_k^{\mathrm{new}}(N) \oplus S_k^{\mathrm{old}}(N),
\end{equation*}
where $S_k^{\mathrm{old}}(N)$ is the space of old forms consisting of cusp forms of level $N$ coming from lower levels as specified by
\begin{equation*}
S_k^{\mathrm{old}}(N)=\operatorname{span}\bigl\{\, f(qz):\, qN'\mid N,\ N'<N,\ f(z)\in S_k(N')\,\bigr\}.
\end{equation*}

The space of newforms $S_k^{\mathrm{new}}(N)$ is defined as the orthogonal complement of $S_k^{\mathrm{old}}(N)$ with respect to the Petersson inner product. This space is stable under the action of Hecke operators, and it is shown in  \cite{Atkin-Li} that these operators can be simultaneously diagonalized. A newform $f$ is called primitive if $\lambda_f(1)=1$; in this case, one has $\psi_f(n)=\lambda_f(n)$ for all $n$. We denote by \(H_k^{*}(N)\) the set of primitive newforms, which forms an orthogonal basis of $S_k^{\mathrm{new}}(N)$. Moreover, for every $f\in H_k^{*}(N)$, the Hecke eigenvalues $\lambda_f(n)$ are real for all $n\in\mathbb{N}$.

Let $k$ and $k'$ be two even positive integers, and let $f\in H_{k}^*(N)$ and $g\in H_{k'}^*(D)$ be primitive forms. We assume that $N$ and $D$ are relatively prime. Given $f$ and $g$ as such, one defines the $L$-series for the Rankin-Selberg convolution as 
\[
L(s, f\otimes g)=\zeta_{ND}(2s)\sum_{n\geq 1}\frac{\lambda_f(n)\lambda_g(n)}{n^{s}},
\] 
where 
\[
\zeta_{ND}(2s)=\sum_{\substack{d\geq 1 \\ (d, ND)=1} } \frac{1}{d^{2s}}=\zeta(2s)\prod_{p|ND}(1-p^{-2s}).
\] 
This series converges absolutely for $\Re(s)>1$.  Setting \[
    \Lambda(s, f\otimes g) = \left(\frac{ND}{4\pi^2}\right)^s\Gamma_g(s) L(s, f \otimes g),
    \]
where
    \begin{equation}\label{eq:Gammag}
    \Gamma_g\left(s\right)
    =\Gamma\left(s+\frac{|k-k'|}{2}\right)
    \Gamma\left(s+\frac{k+k'}{2}-1\right) ,
    \end{equation}
we know that $\Lambda(s, f\otimes g)$ has analytic continuation to $\mathbb{C}$ with a pole at $s=1$ (when $f=g$) and satisfies the functional equation  
\begin{equation}\label{functional-equation}
 \Lambda(s, f\otimes g)=\Lambda(1-s, f\otimes g).
 \end{equation}
The reader is referred to  \cite[Section~4]{KMV RS} for more details.

 In this paper, we study the problem of evaluating the first and second moments of the family
\begin{equation}\label{eqn:family}
\{ L(s,f\otimes g):  f\in H_{k}^*(p^\nu) \},
\end{equation}
where $p$ is a prime and $\nu\geq 3$.  In contrast to the more extensively studied families corresponding to primitive forms of prime or square-free level, $L$-functions of primitive forms of prime power level provide a more localized framework where one can study moments while isolating the effect of increasingly deep ramification at a fixed prime. This perspective is also related to themes appearing in $p$-adic variation, where one studies arithmetic objects as the ramification at a fixed prime varies. While the present work is entirely analytic in nature, the prime power level nonetheless is a natural setting for investigating such fixed prime variation. One distinctive feature here is the richer structure of the space $S_k(p^\nu)$ that brings about contributions from oldforms, which then further leads to a more intricate Petersson trace formula and thereby additional arithmetic considerations in moment computations.

Throughout, the $L$-functions are evaluated at points of the form $s=\frac12+\omega$, where the complex shift $\omega$ is chosen such that
\begin{equation}\label{eqn:shift-condition}
\omega=u+iv\, \text{ with }\,  u, v \in \mathbb{R}, \quad
 |u| \, < \frac{1}{\log(p^\nu)}
 \,\, \text{ and } \,\,
 |\omega|\leq \eta
\end{equation}
for some sufficiently small constant \(\eta>0\). 

Shifted moments play an important role in modern analytic number theory, both because they reveal more transparently the structure of the main terms and because they arise naturally in problems concerning nonvanishing, subconvexity, and correlations of zeros of automorphic L-functions.

All moments considered in this paper are taken with respect to the harmonic weight defined by
\[
\sum_{f\in H_k^*(p^\nu)}^{h} A(f)
=
\sum_{f\in H_k^*(p^\nu)}
\frac{\Gamma(k-1)}{(4\pi)^{k-1}\langle f,f\rangle_{p^\nu}}
A(f),
\]
for any function \(A\).

We establish the following two results concerning the twisted first moment of the family of \(L\)-functions in \eqref{eqn:family}. Our asymptotic formulae include explicit, though nonuniform, error terms with fully specified dependence on all relevant parameters.


\begin{thm}\label{thm:first twisted moment}
Suppose that $k, k'$ are even positive integers with $k\geq8$ and $k>2k'-2$. Let $p$ be a prime, $\ell$ a positive integer with $(\ell, p)=1$, and let $\nu\geq 3$ be an integer. Consider $g\in H_{k'}^*(D)$, where  $D$ is a positive integer such that $D=1$ or $D$ is prime. Let $\omega$ be a complex number satisfying \eqref{eqn:shift-condition}.  Then for any $\epsilon>0$ we have
\begin{align*}
&\sum_{f\in H^*_k(p^\nu)}^{h}
    \lambda_f(\ell) L\Big(\frac12+\omega, f\otimes g\Big)\\&\hspace{2em}
    =  \left(4\pi^2\right)^\omega \Big(1-\frac{1}{p}\Big)^2   \Big(1-\frac{1}{D}\Big) \frac{ \lambda_g(\ell)}{\sqrt\ell} 
       \frac{\Gamma_g\left(\frac12\right)P_g(0)\xi\left(\frac12+\omega\right)^5}{\Gamma_g\left(\frac12+\omega\right)P_g(\omega)\xi\big(\frac12\big)^5}  \left(\frac{(p^\nu D)^\omega-(p^\nu D)^{-\omega}}{2\omega}\right) 
     \\&\hspace{3em}+ \Big(1-\frac{1}{p}\Big) \frac{ \lambda_g(\ell)}{\sqrt\ell} \left( 
     \zeta_{p^\nu D} (1+2\omega)\ell^{-\omega}
  +  \frac{(4\pi^2)^{\omega} \Gamma_g(\frac12-\omega)}{(4\pi^2)^{-\omega} \Gamma_g(\frac12+\omega)}     \zeta_{p^\nu D} (1-2\omega)
   \ell^{\omega}\right)\\&\hspace{3em}
   +O\Big( \ell^{-\frac14+\epsilon} (p^\nu D)^{-\frac{5}{32}+\epsilon}\Big)
   +O\bigg( \frac{D^{A+\frac12+\epsilon-\Re(\omega)} \ell^{A+\frac12+\epsilon}  }{p^{\nu(A+\Re(\omega))} k^{2\Re(\omega)}}\bigg),
\end{align*}
where $A$ is chosen such that $1< A \leq \frac{k-k'}{2}-\epsilon$.
\end{thm}
Taking the limit as $\omega\to0$, we get


\begin{cor}\label{cor:1st-moment}
Assume the hypotheses of Theorem \ref{thm:first twisted moment}.  Then for any $\epsilon>0$ we have
    \begin{align*}
    &\sum_{f\in H_k^*(p^\nu)}^h \lambda_f(\ell) L\Big(\frac12, f\otimes g \Big)
    \\&=
  \left(1-\frac1p\right)^2\left(1-\frac1D\right)\frac{\lambda_g(\ell)}{\sqrt{\ell}}
\left(\nu\log p+\log\frac{e^{2\gamma} D}{4\pi^2\ell}
+\psi\Big(\frac{k-k'+1}{2}\Big)
+\psi\Big(\frac{k+k'-1}{2}\Big)
\right)
\\&\hspace{2em}
+ \Big(1-\frac1p\Big)\frac{\lambda_g(\ell)}{\sqrt{\ell}}\left(\left(1-\frac1D\right)\frac{\log p}{p}
+\Big(1-\frac1p\Big)\frac{\log D}{D}\right) \\
&\hspace{2em}+O\bigg(\frac{ 1}{\ell^{\frac14-\epsilon}(p^\nu D)^{\frac{5}{32}-\epsilon}}\bigg)
  +O\bigg( \frac{D^{A+\frac12+\epsilon} \ell^{A+\frac12+\epsilon} }{p^{\nu A}}\bigg),
   \end{align*}
where $\psi$ is the digamma function, $\gamma$ denotes the Euler-Mascheroni constant, and $A$ is a constant chosen such that $1< A \leq \frac{k-k'}{2}-\epsilon$.
\end{cor}

Closely related to Theorem \ref{thm:first twisted moment} is the work of Bettin \cite{Bettin}, who studied the first moment of twisted Hecke $L$-functions with unbounded shifts $\omega$ in the prime power level aspect, obtaining asymptotic formulae with explicit dependence on the relevant parameters. Our theorem may be viewed as a Rankin-Selberg analogue in which one studies moments of $L(\frac12+\omega,f\otimes g)$ with bounded shifts over primitive forms of prime power level.

Corollary \ref{cor:1st-moment} is similar to a theorem of Pi (see~\cite[Theorem 1.1]{Pi}). In that work, $f$ ranges over primitive forms with nebentypus $\chi_N$, the Kronecker symbol for a prime $N \equiv 1 \pmod{4}$. In such setting, there is no oldform contribution to consider, and the usual Petersson trace formula suffices to evaluate the moment. In contrast, the case of prime power level considered here presents us with a significant oldform contribution, and the argument required to establish the corollary is correspondingly more involved.

As an application of Corollary \ref{cor:1st-moment}, we obtain the following determination result showing that a primitive form $g \in H^*_{k'}(D)$ is uniquely determined by the family of central values $ L\big(\frac12, f\otimes g \big)$ for $f\in {H_k}^*(p^\nu)$. Let $g_1$ and $g_2$ be in $H_{k'_1}^*(D)$, $H_{k'_2}^*(D)$, respectively, such that
\[
 L\Big(\frac12, f\otimes g_1 \Big) =  L\Big(\frac12, f\otimes g_2 \Big) \quad \text{for all }\,\, f\in H_k^*(p^\nu) .
\]
Then by Corollary \ref{cor:1st-moment}, it follows that $\lambda_{g_1}(\ell) = \lambda_{g_2}(\ell)$ for all primes $\ell\neq p$, and thus $g_1=g_2$ by the strong multiplicity one theorem. The reader is referred to \cite{Luo Rama} for the first such determination result.

Next, for the twisted second moment of the family \eqref{eqn:family}, we present the following result. 


\begin{thm}\label{thm:2nd moment}
Let $k, k'$ be even integers satisfying $k+k' >7$, $p$ a prime and $\nu\geq 3$ an integer. Let  $\ell$ be a positive integer such that $(\ell, p)=1$ and $\ell \ll p^{\nu(\frac{2}{17}-\delta)}$ for some sufficiently small $\delta>0$. For a fixed primitive form $g \in H^*_{k'}(1)$ and any $\epsilon>0$, we have
    \begin{align*}
     &\sum^h_{f\in H^*_k(p^\nu)} 
     \lambda_f(\ell) L\Big( \frac12+\omega_1, f\otimes g\Big) 
     L\Big(\frac12+\omega_2, f\otimes g\Big)
    \\
   & =     p^{-\nu(\omega_1+\omega_2)} 
   \Big\{ \mathfrak{m}(\omega_1, \omega_2 ) 
   + \epsilon_{\omega_1}(f\otimes g) \mathfrak{m}(-\omega_1, \omega_2 ) 
   +\epsilon_{\omega_2}(f\otimes g) \mathfrak{m}(\omega_1, -\omega_2 ) 
   \\
   &\qquad \qquad \qquad+\epsilon_{\omega_1}(f\otimes g)\epsilon_{\omega_2}(f\otimes g)
    \mathfrak{m}(-\omega_1, -\omega_2 ) 
   \Big\} \\&\hspace{2em} + O\left(\ell^{\frac{17}{8}}p^{-\nu\left(\frac14-O(\epsilon)\right)}
+\ell^{\frac{5k-5-6\theta}{4k-8\theta}}
p^{-\nu\left(\frac{k-1-2\theta}{4k-8\theta}-O(\epsilon)\right)} +\ell^{\frac{(5-4\theta)(k+k'-2)}{4k-8\theta}}
p^{-\nu\left(\frac{k+k'-2}{4k-8\theta)}-O(\epsilon)\right)}\right)
,
    \end{align*}
where
       \begin{align*}
    &\mathfrak{m}(\omega_1, \omega_2 )
    \\
    =&  \left(1-\frac{1}{p}\right)p^{-\nu\omega_1+\nu\omega_2}\frac{(4\pi^2)^{2\omega_1} }{2\omega_1}  \frac{\Gamma_g(1/2-\omega_1) }{\Gamma_g(1/2+\omega_1)}\frac{L(1-\omega_1+\omega_2, g\otimes g)}{\zeta(2-2\omega_1+2\omega_2)}
    \zeta_{p^\nu}(1-2\omega_1)\zeta_{p^\nu}(1+2\omega_2)\\&\hspace{2em}\times \sum_{de=\ell} 
        \frac{ e^{\omega_1} d^{-\omega_2}}{\sqrt{\ell}}
      \sum_{ab=d} \frac{\mu(a)  \lambda_g(b)}{a^{1-\omega_1+\omega_2} } 
      \Upsilon(1-\omega_1+\omega_2, p, ae) .
    \end{align*}
Here   
\[
    \varepsilon_{\omega}(f \otimes g)
    = \frac{(4\pi^2)^{2\omega}\Gamma_g\big(\frac12-\omega\big)}{\Gamma_g\big(\frac12+\omega\big)},
    \] 
 where $\Gamma_g$ is defined in \eqref{eq:Gammag}. We also have $\xi(s)=s(1-s)\pi^{-s/2}\Gamma\!\left(\frac{s}{2}\right)\zeta(s)$, the completed $\zeta$-function of Riemann. $P_g(s)$ denotes an even polynomial with real coefficients, depending only on $k,k'$, such that
$P_g(s)\Gamma_g\!\left(\frac12+s\right)$ is holomorphic in the region $-A<\Re s$ for some $A>\frac12$. The function $\Upsilon$ is given by 
 \begin{equation}\label{eq:defn upsilon}
    \Upsilon(s, p, ae)
    =  \Big( \sum_{j=0}^\infty
    \frac{\lambda_g(p^j)^2}{p^{sj}}
    \Big)^{-1}
        \prod_{q\mid ae} 
         \Big( \sum_{j=0}^\infty
    \frac{\lambda_g(q^j)^2}{q^{sj}}
    \Big)^{-1}
     \Big( \sum_{j=0}^\infty
    \frac{\lambda_g(q^j) \lambda_g(q^{\nu_{q}(ae)+j})}{q^{sj}}
    \Big).
  \end{equation}
\end{thm}

The method used to prove our theorem also applies to the case of $g \in H_{k'}^{*}(D)$ for square-free $D.$ However, carrying this out would involve a substantial amount of technical computations without adding new conceptual ingredients. We therefore do not pursue this extension here. On another note, Theorem \ref{thm:2nd moment} applied with $\omega_2 = \overline{\omega_1}=\omega$ can be used to establish subconvexity and nonvanishing results for the family of values $\big\{L\big(\frac12+\omega, f\otimes g \big)\big\}$.

Letting $\omega_1, \omega_2$ tend to $0$ in Theorem \ref{thm:2nd moment}, we obtain 


\begin{cor}\label{cor:thm 2nd moment}
Under the hypotheses of Theorem \ref{thm:2nd moment}, we have
\begin{align*}
\sum_{f\in H_k^\star(p^\nu)}^h
\lambda_f(\ell)
L\Big(\frac12,f\otimes g\Big)^2
&=
Q_{g,\ell,p}(\nu\log p)\\&\hspace{-6em}
+
O\left(\ell^{\frac{17}{8}}p^{-\nu\left(\frac14-O(\epsilon)\right)}
+\ell^{\frac{5k-5-6\theta}{4k-8\theta}}
p^{-\nu\left(\frac{k-1-2\theta}{4k-8\theta}-O(\epsilon)\right)}
+
\ell^{\frac{(5-4\theta)(k+k'-2)}{4k-8\theta}}
p^{-\nu\left(\frac{k+k'-2}{4k-8\theta}-O(\epsilon)\right)}
\right),
\end{align*}
where $Q_{g,\ell,p}$ is a polynomial of degree $3$ with leading coefficient
\[
\frac{2}{\pi^2}
\left(1-\frac1p\right)^3
L(1,\operatorname{sym}^2 g)
\frac1{\sqrt{\ell}}
\sum_{r\mid \ell}
\lambda_g\left(\frac{\ell}{r}\right)
\frac{\phi(r)}{r}
\Upsilon(1,p,r).
\]
Here $L(s,\operatorname{sym}^2 g)$ denotes the symmetric square $L$-function of $g$, $\phi$ is the Euler's totient function and $\Upsilon$ is as defined in \eqref{eq:defn upsilon}.

\end{cor}


\begin{rem*}
If $\ell$ is square-free, then
\[
\sum_{r\mid \ell}
\lambda_g\left(\frac{\ell}{r}\right)
\frac{\phi(r)}{r}
\Upsilon(1,p,r) 
=
\Upsilon(1,p,1)\lambda_g(\ell)
\prod_{r\mid \ell}\frac{2r}{r+1}.
\]
In particular, for $\ell=1$, the leading coefficient becomes
\[
\frac{2}{\pi^2}
\left(1-\frac1p\right)^3
\Upsilon(1,p,1)
L(1,\operatorname{sym}^2 g).
\]
The remaining coefficients of the polynomial $Q_{g,\ell,p}$  are computed explicitly using Maple. Since their expressions are lengthy, they will not be recorded here.
\end{rem*}

Moments of automorphic $L$-functions occupy a central place in analytic number theory, with applications ranging from nonvanishing and subconvexity problems to questions concerning low-lying zeros and arithmetic statistics. One of the earliest breakthroughs in moments of $L$-functions of automorphic forms is due to Duke \cite{Duke}, who proved that as $p\to\infty$ among prime numbers, we have
\[
\frac{1}{|H_2^*(p)|}\sum_{f\in H_2^*(p)} L\Big(\tfrac12,f\Big)^2
= \log p + O\big(p^{-2}\log p\big).
\]
Since the family $\displaystyle\{L\big(1/2,f\big)\}_{f\in H_2^*(p)}$ forms an orthogonal family, the main term on the right-hand side agrees with the prediction of Conjecture~3 in \cite{BK}, which was derived from a hybrid product expression for the $L$-function. Note that Akbary \cite{Akbary} later generalized Duke's result to forms of arbitrary even weight and prime level $p$.

The fourth moment was computed by Kowalski, Michel, and VanderKam \cite{KMV Mol}. As $p\to\infty$ through primes, they showed that
\begin{equation}\label{eq:KVM 4th moment}
\sum_{f\in H_2^*(p)}^h L\Big(\tfrac12,f\Big)^4
= \mathcal{Q}(\log p) + O_\varepsilon\big(p^{-\frac1{12}+\varepsilon}\big),
\end{equation}
where $\mathcal{Q}$ is a sextic polynomial whose leading coefficient is $\frac{1}{60\pi^2}$. Balkanova and Frolenkov \cite{BF} later improved the error term to $O_\varepsilon(p^{-25/228+\varepsilon})$. A straightforward computation, which we omit here, confirms that the leading term again agrees with the prediction of Conjecture~3 in \cite{BK}.

A result that is particularly relevant to the present work is by Balkanova \cite[Corollary 1.5]{Balkanova}, which is an analogue of \eqref{eq:KVM 4th moment} for primitive forms $f$ of even weight $k$ and prime power level $p^\nu$, where $p$ is a fixed prime and $\nu\ge 3$. More precisely, Balkanova proved in \cite{Balkanova}  that
\[
\sum_{f\in H_k^*(p^\nu)}^h
L\Big(\tfrac12,f\Big)^4
=
\CMcal{Q}_k\big(\nu\log p\big)
+
O_{\epsilon,k,p}\Big(
p^{-\nu\left(\frac14-\epsilon\right)}
+
p^{-\nu\left(\frac{k-1-2\theta}{8-8\theta}-\epsilon\right)}
\Big).
\]
Here $\CMcal{Q}_k$ is a sextic polynomial depending on $k$ whose leading coefficient is
\[
\frac{\phi(p^\nu)^7}{p^{7\nu}}
\frac{p^2-1}{p^2}
\frac{1}{60\pi^2}.
\]
The  proof relies on an extension of the Petersson trace formula for $f\in H_k^*(p^\nu)$ established by Rouymi \cite{Rouymi}. In the same work, Rouymi also obtained asymptotics for twisted first, second, and third moments at the central point over $f\in H_k^*(p^\nu)$.

We now turn to results concerning $L$-functions of Rankin-Selberg convolutions. In \cite{KMV RS}, Kowalski, Michel, and Vanderkam showed that for a fixed $g\in H^*_{k'}(D)$ with $D$ square-free and $(D,p)=1$, and for even integers $k,k'$ with $k<12$,  we have
\begin{align*}
\sum_{f\in H_k^*(p)}^h
\Big|L\Big(\tfrac12,f\otimes g\Big)\Big|^2
=\mathcal{Q}_g(\log p)
+O_{g,k,\epsilon}\big(p^{-\frac{1}{12}+\epsilon}\big) \,\, \text{ for every }\, \epsilon>0,
\end{align*}
as $p\to\infty$ in the primes. Here $\mathcal{Q}_g$ is a cubic polynomial depending on $g$ whose leading coefficient is
\[
\frac{L(1,\operatorname{sym}^2 g)}{3\zeta(2)}
\prod_{q\mid D}\frac{(q-1)^2}{q(q+1)}.
\]
(We believe there is a minor typo in \emph{loc.\ cit.}) This leading coefficient may be compared with the one in our Corollary~\ref{cor:thm 2nd moment}.


\noindent \textbf{Notation.} 
Throughout the paper, $p$ will denote a prime number and $\nu$ a positive integer that is at least $3$. We also use $\ell$ to denote a positive integer indivisible by $p.$ The variables $\omega$,  $\omega_1$ and $\omega_2$ will be complex numbers used to shift the $L$-values under consideration away from the central point $s=\frac12$. Throughout the paper, $A$ denotes a positive constant whose value may vary from line to line.

Given two functions $f(x)$ and $g(x)$, we shall interchangeably use the notation  $f(x)=O(g(x))$ and $f(x) \ll g(x)$  to mean that there exists $ \mathcal{C} >0$ such that $|f(x)| \le \mathcal{C} |g(x)|$ for all sufficiently large $x$. Sometimes we will use the notation $f(x)\ll_v  g(x)$ when the constant $ \mathcal{C}$ depends on the parameter $v.$ Most of the constants implicit in the big-$O$ notations in Section \ref{sec:second moment} and afterwards may depend on the weights of forms $k,k'$ and the prime $p$, although we sometimes drop the subscript to simplify the notation.


\section{Preliminaries}
\subsection{Kloosterman Sums} A Kloosterman sum is defined as
    \begin{equation}\label{eqn:kloosterman}
    S(m, n ; c) = \sum_{\substack{1\leq x \leq c, \\ (x, c)=1}} e\Big(\frac{xm+\overline{x}n}{c}\Big) .
    \end{equation}
Clearly, $S(m, n ; c)= S(n, m ; c)$. This sum also has the following property, which we will later need.


\begin{lem}{\cite[Lemma A.12]{Royer}}\label{Kloosterman vanishes}
Let $m, n, c$ be positive integers and $q$ be a prime. Suppose that $q^2 \mid c$, $q\mid m$ and $q \nmid n$. Then $S(m, n; c)=0$.
\end{lem}
We also use Weil's bound which states that
\begin{equation}\label{eqn:weil-bound} 
S(m, n; c) \ll 
(m, n, c)^{\frac12} \tau(c) \sqrt{c}.
\end{equation}

\subsection{Hecke Relations} Next, we note some results on normalized Fourier coefficients of primitive forms. 


\begin{lem}[Hecke recursion formula]\label{Hecke relation}
For a normalized primitive form $f\in H^*_{k}(N)$, we have
    \[
    \lambda_f(mn)= \sum_{\substack{d\mid (m, n), \\ (d, N)=1}} \mu(d) 
    \lambda_f\Big(\frac{m}{d}\Big) \lambda_f\Big(\frac{n}{d}\Big)
    \]
and
    \[
    \lambda_f(m) \lambda_f(n)= \sum_{\substack{d\mid (m, n), \\ (d, N)=1}}
    \lambda_f\Big(\frac{mn}{d^2}\Big).
    \]
\end{lem}

For a proof of this lemma, see \cite[(14.50) and (14.51)]{IK} (note that there is a typo in \textit{loc cit} where $\lambda_g(mn)$ and $\lambda_g(m) \lambda_g(n)$ should be swapped). A useful consequence of this, which can be proven by the inclusion-exclusion principle~\cite[p. 135]{Balkanova}, is the following. For any function $H$, we have
    \begin{equation}\label{eq:n p rel prime}
    \sum_{n\geq 1, (n, p)=1} \lambda_g(n) H(n)
    =\sum_{n\geq 1} \lambda_g(n) H(n) 
    -\lambda_g(p) \sum_{n\geq 1} \lambda_g(n) H(np)
    + \sum_{n\geq 1} \lambda_g(n) H(np^2).
    \end{equation}
We will apply this formula frequently to simplify sums of Fourier coefficients over $n$ with $(n, p)=1$.


\subsection{Summation Formulae and Estimates}
We begin by stating the following variation of Petersson Trace Formula due to Rouymi.


\begin{thm}[Petersson Trace Formula and Remark 4 of \cite{Rouymi}]\label{Rouymi}
Let $p$ be a prime and $\nu\geq 3$ a positive integer. Then
    \[
    \Delta^*_{p^\nu}\big( a, b\big)
    =
    \sum_{f \in H^*_k(p^\nu)}^{h} \lambda_f(a) \lambda_f(b)
    =
    \begin{cases}
    \Delta_{p^\nu}\big( a, b\big) 
    -\sdfrac{\Delta_{p^{\nu-1}}\big( a, b\big)}{p} 
    \quad &\text{if } \, (ab, p^\nu)=1, \\
    0 \quad &\text{else},
    \end{cases}
    \]
and
    \[
    \Delta_{p^\nu}\big( a, b\big) 
    = \sum_{f \in H_k(p^\nu)}^{h} \lambda_f(a) \lambda_f(b) 
    = \delta_{a, b} +
    2\pi i^{-k} \sum_{p^\nu \mid c} \frac{S(a, b; c)}{c}
    J_{k-1}\Big(\frac{4\pi \sqrt{ab}}{c} \Big).
    \]
Here,  $J_{k-1}(z)$ denotes the Bessel function of the first kind of order $(k-1)$.
\end{thm}


Next, we state a version of Vorono\"{\i} Summation formula simplified to our setting.

\begin{thm}[Vorono\"{\i} Summation Formula, see p.~185 in \cite{KMV RS}]
 \label{Voronoi}
Consider integers $a$ and $c$ with $(a, c)=1$, and let $g \in H^*_{k'}(1)$ be a primitive form. For a smooth function $H\in C^\infty(0,\infty)$ vanishing in a neighbourhood of zero and rapidly decreasing we have
    \[
    c\sum_{n \geq 1}
    \lambda_g(n) e\Big(\frac{an}{c}\Big) H(n)
    = 
     \sum_{n\geq 1} \lambda_{g}(n) e\Big(\frac{-n\overline{a}}{c}\Big)
    \int_0^\infty H(x)  J_{g}\Big(\frac{4\pi\sqrt{nx}}{c}\Big)\mathop{dx},
    \]
    where $\overline{a}$ is the multiplicative inverse of $a$ modulo $c$, and 
$J_g(x)= 2\pi i^{k'} J_{k'-1}(x)$.

\end{thm}


The following is a  large sieve inequality due to Deshouillers and Iwaniec  \cite[Theorem 9]{DI} and  Matom\"{a}ki \cite[Lemma 9]{Matomaki}.

\begin{thm}[Theorem 9 in \cite{DI} and Lemma 9 in \cite{Matomaki}] \label{Matomaki}
Let $r, s, d$ be positive integers that are pairwise coprime. Let $H$ be an infinitely differentiable function that is supported on $[M, 2M]\times [N, 2N] \times [C, 2C]$ for $M, N, C>0$ such that
    \[ 
    \bigg| \frac{\partial^{j+k+\ell}}{{\partial m}^j {\partial n}^k {\partial c}^\ell} \, H(m, n, c) \bigg| 
     \leq M^{-j} N^{-k} C^{-\ell} \quad
     \text{for } \, 0\leq j, k, \ell \leq 2.
    \]
Set $X_{d}= \frac{\sqrt{dMN}}{sC\sqrt{r}}$. Then for any $\epsilon>0$ and complex sequences $\textbf{a}=(a_m)$ and $\textbf{b}=(b_n)$ we have
    \begin{align*}
    & \sum_m a_m 
    \sum_n b_n 
    \sum_{c, (c, r)=1} H(m, n, c)
    S(d m \overline{r}, \pm n, sc) \\
    \ll
    &\, s\sqrt{r} C^{1+\epsilon} {d}^{\frac{7}{64}}
    \frac{\big(1+{X_{d}}^{-1}\big)^{\frac{7}{32}}}{1+X_{d}}
    \bigg(1+X_{d}+\sqrt{\frac{M}{r s}}\bigg)
    \bigg(1+X_{d}+\sqrt{\frac{N}{r s}}\bigg)
    ||\bf{a}||_2 ||\bf{b}||_2 ,
    \end{align*}
where $||\bf{a}||_2  = \sqrt{\sum_{m} |a_m|^2}$. 
\end{thm}

The following result is ~\cite[Proposition B.1]{KMV RS} which is a generalization of \cite[Theorem 1]{DFI}.  It provides an upper bound for a shifted convolution sum of the Fourier coefficients of a primitive form.


\begin{thm}\label{DFI type result}
Let $g$ be a primitive form of square-free level with trivial nebentypus. Suppose that $H$ is a smooth test function on $\mathbb{R}^+ \times \mathbb{R}^+$ that satisfies the condition
    \begin{equation}\label{eq:DFI H condition}
    z^j y^k H^{(jk)}(z, y)
    \leq_{j, k} 
    \Big(1+\frac{z}{Z}\Big)^{-A}
    \Big(1+\frac{y}{Y}\Big)^{-A}
    P^{j+k}
    \quad \text{for all integers } \, j, k \geq 0 , \,\, \text{for all } A>0
    \end{equation}
for some $Z, Y, P \geq 1$. Let $(a, b)=1$ and $h\neq 0$. Then 
    \begin{equation}\label{eq:DFI smoothed sum}
    D_H^{\pm}(a, b ; h)
    = \sum_{a m\pm b n=h }
    \lambda_g(m) \lambda_g(n) H(a m, b n)
    \ll_{g, \epsilon}
    \,
    P^{\frac54} (Z+Y)^{\frac14}
    (ZY)^{\frac14+\epsilon}.
    \end{equation}
\end{thm}


\subsection{Asymptotic Estimates for Bessel and Gamma Functions} We now note some formulae and estimates on Bessel and Gamma functions. 

Assume that $s=\sigma+it$ satisfies $|s|\to\infty$ and $|\arg s|\, \le \pi-\delta$ for some fixed $\delta>0$. A convenient uniform form of Stirling's formula for the magnitude of the Gamma function is
\begin{equation}\label{eqn: stirling-formula}
\log{|\Gamma(s)|}
=
\Big(\sigma-\tfrac12\Big)\log|s|
-\sigma
- t\,\arg(s)
+\tfrac12\log(2\pi)
+O_\delta\!\left(\frac1{|s|}\right).
\end{equation}

The exponential decay is governed by the term $-t\,\arg(s)$. If $|t|$ dominates and
$\sigma$ is moderate, then $\arg(s)\approx \pm \pi/2$ and one recovers the
usual decay factor $e^{-\pi|t|/2}$. If instead $\sigma$ is comparable to or larger
than $|t|$, then $\arg(s)$ is small and the factor $e^{-\pi|t|/2}$ is no longer the
correct description; in this regime the behavior is better expressed in terms of
$|s|^{\sigma-\frac12}e^{-\sigma}$, with the contribution of the term
$-t\,\arg(s)$ retained. 

Using this version of the Stirling's formula, we derive the following result for the ratio of $\Gamma$-values that appears in the functional equation of $L(s,f\otimes g)$.


\begin{lem}\label{lem:log derivative gamma g bound}
Let $k,k',$ and A be constants such that $k > k'>0$ and $0<A < \frac{k-k'+1}{2}$. Let $\omega$ be a complex number with $|\omega|\leq \eta$ for sufficiently small $\eta$. For any real $t$ we have
    \[
    \frac{\Gamma_g\big(\frac12+A+it\big) }{\Gamma_g\big(\frac12+\omega\big) }
    \ll_{\omega}
    k^{-2\Re(\omega)-k+1} \, \big(k^2+t^2\big)^{A+\frac{k-1}{2}}.
    \]
\end{lem}

To later bound general Bessel functions of the first kind, we record the following result.

\begin{lem}[Lemma C2 in \cite{KMV RS}]\label{lem:Bessel}
For $z>0$ and $j_1 \geq 0$, 
    \[
    \frac{z^{j_1}}{(1+z)^{j_1}}J^{(j_1)}_{j_2}(z) \ll_{j_1, j_2}
    \frac{z^{\Re(j_2)}}{(1+z)^{\Re(j_2)+\frac12}}. 
    \]
\end{lem}


\subsection{The Approximate Functional Equation}
By~\cite[(4.9)]{KMV RS}, for $\omega$ satisfying \eqref{eqn:shift-condition}, we have
    \begin{equation}\label{eq:AF}
    \begin{split}
    L\Big(\frac12+\omega, f\otimes g\Big) 
    =&\, ( p^{ \omega} D)^{-\nu} \sum_{n\geq 1}  \frac{\lambda_f(n) \lambda_g(n)}{\sqrt n} 
    V_{g, \omega}\Big(\frac{n}{p^\nu D}\Big)
    \\
    &+( p^{ \omega} D)^{-\nu} \varepsilon_{\omega}(f \otimes g)
    \sum_{n\geq 1} 
    \frac{\lambda_f(n) \lambda_g(n)}{\sqrt n} 
    V_{g, -\omega}\Big(\frac{n}{p^\nu D}\Big).
    \end{split}
    \end{equation}
Here 
  \[
    \varepsilon_{\omega}(f \otimes g)
    = \frac{(4\pi^2)^{2\omega}\Gamma_g\big(\frac12-\omega\big)}{\Gamma_g\big(\frac12+\omega\big)},
    \] and 
    \begin{align}\label{eqn:V-g-omega}
    V_{g, \omega}(y)
    &= \frac{1}{2\pi i} \int_{(3)} H_{g, \omega}(s)
     \zeta_{p^\nu D} (1+2s)  y^{-s} \frac{\mathop{ds}}{s-\omega},    \end{align}
where 
    \[
    H_{g, \omega}(s)
    =(4\pi^2)^{-s} \Gamma_g\Big(\frac12+s\Big) G_{g, \omega}(s) ,
    \]
with
    \[
    \Gamma_g(s)= 
    \Gamma\Big(s+\frac{|k-k'|}{2}\Big)
    \Gamma\Big(s+\frac{k+k'}{2}-1\Big)
    \]
and
    \[
    G_{g, \omega}(s)=
    \frac{(4\pi^2)^{\omega}}{\Gamma_g\big(\frac12+\omega\big)}
    \frac{\xi\big(\frac12+s-\omega\big)^5}{\xi\big(\frac12\big)^5} \frac{P_g(s)}{P_g(\omega)} .
    \]
We recall that $\xi$ is the completed Riemann zeta function and $P_g(s)$ is chosen as an even polynomial wit real coefficients such that $P_g(s) \Gamma_g\big(\frac12+s\big)$ is holomorphic in a region $\Re(s) > -A$ for some $A>\frac12$. For simplicity, we normalize $P_g$ so that $P_g(0)=1$.

Equation \eqref{eq:AF} is a crucial ingredient in our work as it allows us to write the central values $L\left(\frac12+\omega,f\otimes g\right)$ in terms of  rapidly decaying series built from the Fourier coefficients of $f$ and $g$. 
\begin{lem}\label{V bounds}
Let $\omega\in\mathbb{C}$ with $|\Re(\omega)| < \frac{1}{\log(p^\nu)}$. For all $A>0$, we have
    \[
    V_{g, \omega}(x) \ll_A \big(1+|\Im(\omega)|\big)^B x^{-A}.
    \]
Further,
    \[
    V_{g, \omega}(x)
    =    \big(\mathop{\mathrm{Res}}_{s=0}+\mathop{\mathrm{Res}}_{s=\omega}\big) 
    \bigg(H_{g,\omega}(s)  \zeta_{p^\nu D} (1+2s) \frac{x^{-s}}{s-\omega} \bigg)
    + O_{\epsilon, g}\Big((1+|\Im(\omega)|)^B (p^\nu D)^{\frac{3}{32}+\epsilon} x^{\frac14}\Big).
    \]
That is,
        \begin{align*}
    V_{g, \omega}(x)
    =&  
    -\frac{(4\pi^2)^{\omega}} {2\omega} \Big(1-\frac1{p}\Big)  \Big(1-\frac1{D}\Big) \frac{\Gamma_g\left(\frac12\right)P_g(0) \xi\big(\frac12-\omega\big)^5}{\Gamma_g\left(\frac12+\omega\right)P_g(\omega)
    \xi\big(\frac12\big)^5} 
    \\
    &+ \zeta_{p^\nu D} (1+2\omega) x^{-\omega}
    + O_{\epsilon, g}\Big((1+|\Im(\omega)|)^B (p^\nu D)^{\frac{3}{32}+\epsilon} x^{\frac14}\Big). 
    \end{align*}
\end{lem}


\begin{proof}
This follows from \cite[Eq.~4.12, Eq.~4.14]{KMV RS}. The value of the first residue is
    \begin{align*}
     \mathop{\mathrm{Res}}_{s=0}
    & \bigg\{H_{g,\omega}(s) \Big(1-\frac{1}{p^{1+2s}}\Big) \Big(1-\frac{1}{D^{1+2s}}\Big) \zeta(1+2s) \frac{x^{-s}}{s-\omega} \bigg\} 
    \\
    &=
     -\frac{1}{2\omega}\Big(1-\frac1{p}\Big) \Big(1-\frac1{D}\Big) H_{g,\omega}(0) 
    = -\frac{1}{2\omega} \Big(1-\frac1{p}\Big)  \Big(1-\frac1{D}\Big) \Gamma_g\Big(\frac12\Big) G_{g, \omega}(0)
    \\
    & = -
    \frac{(4\pi^2)^{\omega}} {2\omega} \Big(1-\frac1{p}\Big) \Big(1-\frac1{D}\Big) \frac{\Gamma_g\left(\frac12\right)}{\Gamma_g\left(\frac12+\omega\right)}
    \frac{\xi\big(\frac12-\omega\big)^5}{\xi\big(\frac12\big)^5} \frac{P_g(0)}{P_g(\omega)} .
    \end{align*}
The other residue can easily be found to be $\zeta_{p^\nu D}(1+2\omega)x^{-\omega}$ by noting that $H_{g,\omega}(\omega)=1$.
\end{proof}


\section{The Twisted First Moment}\label{sec:first moment}

In this section, we let $g$ be a primitive form  in $H^*_{k'}(D)$, where $D=1$ or $D$ is a prime number distinct from the prime $p$. Let $\ell$ be a positive integer such that $p\nmid \ell$. We define 
    \[
    M_g(p^{\nu},\ell, \omega ;1)
    := \sum_{f\in H^*_k(p^\nu)}^{h}
    \lambda_f(\ell)
    L\left(\frac12+\omega, f\otimes g\right).
    \]
Here we recall that $k\geq8$ is an even integer and $\omega$ is complex number satisfying \eqref{eqn:shift-condition}. 

We seek an asymptotic forula for  $M_g(p^{\nu},\ell, \omega ;1)$ in the $\nu$-aspect. By \eqref{eq:AF}, we have
    \begin{equation*} 
    M_g(p^{\nu},\ell, \omega;1)
    = \mathcal{M}_1(\omega)+\frac{(4\pi^2)^{\omega} \Gamma_g(\frac12-\omega)}{(4\pi^2)^{-\omega} \Gamma_g(\frac12+\omega)}  \mathcal{M}_2(\omega),\end{equation*}
    where it is set that
    \begin{align*}
    \mathcal{M}_1(\omega)&:=(p^\nu D)^{- \omega}  \sum_{n \geq 1}
    \frac{\lambda_g(n) }{\sqrt{n}} 
    V_{g, \omega}\Big(\frac{n}{p^\nu D}\Big)
     \sum^h_{f\in H^*_k(p^\nu)}  \lambda_f(n) \lambda_f(\ell),  \end{align*}
     and
     \begin{align*}
    \mathcal{M}_2(\omega) &:=
    (p^\nu D)^{-\omega}  \sum_{n \geq 1}
    \frac{\lambda_g(n) }{\sqrt{n}} 
    V_{g, -\omega}\Big(\frac{n}{p^\nu D}\Big)
     \sum^h_{f\in H^*_k(p^\nu)}  \lambda_f(n) \lambda_f(\ell).
     \end{align*}
 We compute $\mathcal{M}_1(\omega)$ first, as $ \mathcal{M}_2(\omega)$ can be computed similarly. By Theorem \ref{Rouymi}, since $p\nmid \ell$, we can write
    \[
     \mathcal{M}_1(\omega)
     =  \mathcal{M}_1^D(\omega)
     + \mathcal{M}_{1,1}^{ND}(\omega)+ \mathcal{M}_{1,2}^{ND}(\omega),
    \]
where
    \begin{align*}
      \mathcal{M}_1^D(\omega)
     :=  \Big(1-\frac{1}{p}\Big) \frac{(p^\nu D)^{-\omega} \lambda_g(\ell)}{\sqrt\ell} 
    V_{g, \omega}\Big(\frac{\ell }{p^\nu D}\Big), 
    \end{align*}
    \begin{align*}
     \mathcal{M}_{1,1}^{ND}(\omega)
    := 2\pi i^{-k} (p^\nu D)^{-\omega}  
     \sum_{\substack{n\geq 1, \\ (n, p)=1}}
    \frac{\lambda_g(n)}{\sqrt n} 
    V_{g, \omega}\Big(\frac{n}{p^\nu D}\Big)
    \sum_{p^\nu \mid c} 
    \frac{S\big(n, \ell; c\big)}{c}
    J_{k-1}\Big(\frac{4\pi\sqrt{n\ell}}{c}\Big),
    \end{align*}
and 
    \begin{align*}
     \mathcal{M}_{1,2}^{ND}(\omega)
    := -\frac{2\pi i^{-k} (p^\nu D)^{-\omega}}{p}   
     \sum_{\substack{n\geq 1\\ (n,p)=1}}
    \frac{\lambda_g(n)}{\sqrt n} 
    V_{g, \omega}\Big(\frac{n}{p^\nu D}\Big)
    \sum_{p^{\nu-1} \mid c} 
    \frac{S\big(n, \ell; c\big)}{c}
    J_{k-1}\Big(\frac{4\pi\sqrt{n\ell}}{c}\Big).
    \end{align*}
\subsection{Computing the Diagonal Contribution} Using Lemma \ref{V bounds} gives  
  \begin{align*}
    \mathcal{M}_1^D(\omega)
       &= \mathcal{M}_{1, 1}^D(\omega) + \mathcal{M}_{1, 2}^D(\omega)
    + O_{\epsilon, g}\Big((1+|\Im(\omega)|)^B \ell^{-\frac14+\epsilon} (p^\nu D)^{-\frac{5}{32}+\epsilon}\Big)
     \end{align*}
     with 
     \begin{equation*}
     \mathcal{M}_{1, 1}^D(\omega) :=- \Big(1-\frac{1}{p}\Big)^2  \Big(1-\frac1{D}\Big)  \frac{ \lambda_g(\ell)}{\sqrt\ell} 
        \frac{(p^\nu D)^{-\omega} (4\pi^2)^{\omega}} {2\omega} \frac{\Gamma_g\left(\frac12\right)P_g(0) \xi\big(\frac12-\omega\big)^5}{\Gamma_g\left(\frac12+\omega\right)P_g(\omega)
    \xi\big(\frac12\big)^5}
     \end{equation*}
     and 
     \begin{equation*}
     \mathcal{M}_{1, 2}^D(\omega) := \Big(1-\frac{1}{p}\Big) \frac{ \lambda_g(\ell)}{\sqrt\ell} (p^\nu D)^{-\omega} 
     \zeta_{p^\nu D} (1+2\omega)
    \Big(\frac{\ell}{p^\nu D}\Big)^{-\omega}.
     \end{equation*}
 Similarly, we get
     \begin{align*}
    \mathcal{M}_2^D(\omega)
     &= \mathcal{M}_{2, 1}^D(\omega) + \mathcal{M}_{2, 2}^D(\omega)
    + O_{\epsilon, g}\Big((1+|\Im(\omega)|)^B \ell^{-\frac14+\epsilon} (p^\nu D)^{\left(-\frac{5}{32}+\epsilon\right)}\Big),
     \end{align*}
     with $ \mathcal{M}_{2, 1}^D(\omega)= \mathcal{M}_{1, 1}^D(-\omega)$ and $ \mathcal{M}_{2, 2}^D(\omega) =   \mathcal{M}_{1, 2}^D(-\omega)$.
 By setting
     \[
    M_g^D(p^{\nu},\ell, \omega;1)
     := \mathcal{M}_1^D(\omega) + \frac{(4\pi^2)^{\omega} \Gamma_g(\frac12-\omega)}{(4\pi^2)^{-\omega} \Gamma_g(\frac12+\omega)}  \mathcal{M}_2^D(\omega),
     \]
we get
\begin{align*}
M_g^D(p^{\nu},\ell, \omega;1)&=  \left( \mathcal{M}_{1, 1}^D(\omega) + \mathcal{M}_{1, 2}^D(\omega)\right)+ \frac{(4\pi^2)^{\omega} \Gamma_g(\frac12-\omega)}{(4\pi^2)^{-\omega} \Gamma_g(\frac12+\omega)} \left(\mathcal{M}_{2, 1}^D(\omega) 
+ \mathcal{M}_{2, 2}^D(\omega)\right)
+\\&\quad+ O_{\epsilon, g}\Big((1+|\Im(\omega)|)^B \ell^{-\frac14+\epsilon} (p^\nu D)^{\left(-\frac{5}{32}+\epsilon\right)}\Big).
\end{align*}
We have 
  \begin{align*}
     \mathcal{M}_{1, 1}^D(\omega)
    + \frac{(4\pi^2)^{\omega} \Gamma_g(\frac12-\omega)}{(4\pi^2)^{-\omega} \Gamma_g(\frac12+\omega)}  \mathcal{M}_{2, 1}^D(\omega)
    & =      
      \left(4\pi^2\right)^\omega \Big(1-\frac{1}{p}\Big)^2   \Big(1-\frac{1}{D}\Big) \frac{ \lambda_g(\ell)}{\sqrt\ell} 
       \frac{\Gamma_g\left(\frac12\right)P_g(0)\xi\left(\frac12+\omega\right)^5}{\Gamma_g\left(\frac12+\omega\right)P_g(\omega)\xi\big(\frac12\big)^5}
       \\
     &\quad \times \left(\frac{-(p^\nu D)^{-\omega}+(p^\nu D)^\omega}{2\omega}\right),     \end{align*}
and 
  \begin{align*}
    & \mathcal{M}_{1, 2}^D(\omega)
    + \frac{(4\pi^2)^{\omega} \Gamma_g(\frac12-\omega)}{(4\pi^2)^{-\omega} \Gamma_g(\frac12+\omega)}  \mathcal{M}_{2, 2}^D(\omega)
    \\& =     \Big(1-\frac{1}{p}\Big) \frac{ \lambda_g(\ell)}{\sqrt\ell} \left( 
     \zeta_{p^\nu D} (1+2\omega)\ell^{-\omega}
  +  \frac{(4\pi^2)^{\omega} \Gamma_g(\frac12-\omega)}{(4\pi^2)^{-\omega} \Gamma_g(\frac12+\omega)}     \zeta_{p^\nu D} (1-2\omega)
   \ell^{\omega}\right).
    \end{align*}
Hence, the diagonal contribution is as follows.
\begin{align}\label{eqn:1st-moment-D}
M_g^D(p^{\nu},\ell, \omega;1)&=  \left(4\pi^2\right)^\omega \Big(1-\frac{1}{p}\Big)^2   \Big(1-\frac{1}{D}\Big) \frac{ \lambda_g(\ell)}{\sqrt\ell} 
       \frac{\Gamma_g\left(\frac12\right)P_g(0)\xi\left(\frac12+\omega\right)^5}{\Gamma_g\left(\frac12+\omega\right)P_g(\omega)\xi\big(\frac12\big)^5}
        \left(\frac{-(p^\nu D)^{-\omega}+(p^\nu D)^\omega}{2\omega}\right)\nonumber \\&\quad+ \Big(1-\frac{1}{p}\Big) \frac{ \lambda_g(\ell)}{\sqrt\ell} \left( 
     \zeta_{p^\nu D} (1+2\omega)\ell^{-\omega}
  +  \frac{(4\pi^2)^{\omega} \Gamma_g(\frac12-\omega)}{(4\pi^2)^{-\omega} \Gamma_g(\frac12+\omega)}     \zeta_{p^\nu D} (1-2\omega)
   \ell^{\omega}\right)\\&\quad+O_{\epsilon, g}\Big((1+|\Im(\omega)|)^B \ell^{-\frac14+\epsilon} (p^\nu D)^{-\frac{5}{32}+\epsilon}\Big).\nonumber
\end{align}


\subsection{Estimating the Nondiagonal Contribution}

We now bound the nondiagonal terms $ \mathcal{M}_{1,1}^{ND}(\omega)$. Estimations of $ \mathcal{M}_{1,2}^{ND}(\omega)$, $ \mathcal{M}_{2,1}^{ND}(\omega)$ and $ \mathcal{M}_{2,2}^{ND}(\omega)$ can then be written down with suitable modifications.
   
 Recall that 
  \begin{align}\label{eqn:M11ND}
     \mathcal{M}_{1,1}^{ND}(\omega)
    = 2\pi i^{-k} (p^\nu D)^{-\omega}  
     \sum_{\substack{n\geq 1, \\ (n, p)=1}}
    \frac{\lambda_g(n)}{\sqrt n} 
    V_{g, \omega}\Big(\frac{n}{p^\nu D}\Big)
    \sum_{p^\nu \mid c} 
    \frac{S\big(n, \ell; c\big)}{c}
    J_{k-1}\Big(\frac{4\pi\sqrt{n\ell}}{c}\Big).
 \end{align}
The Bessel function $J_{k-1}$ can be written as the inverse Mellin transform     
\begin{equation}\label{eqn:bessel-mellin}
    J_{k-1}(u) = \frac{1}{2\pi i} \int_{(\alpha)} \frac{\Gamma\big(\frac{k-1+z}{2}\big)}{\Gamma\big(\frac{k+1-z}{2}\big)}
    \Big(\frac{u}{2}\Big)^{-z} \mathop{dz}  \qquad \text{for } \,\, 1-k < \alpha <0. 
\end{equation}
In our application, we set $\alpha= -\frac12-\delta$ for a positive $\delta$ to be chosen later, so that $\delta <k-\frac32 $. By substituting \eqref{eqn:V-g-omega} and \eqref{eqn:bessel-mellin} into \eqref{eqn:M11ND}, we get
    \begin{equation}\label{eq:N ND 11}
    \begin{split}
     \mathcal{M}_{1,1}^{ND}(\omega)
    = &\, \frac{2\pi i^{-k}}{(2\pi i)^2}\frac{ (p^\nu D)^{-\omega} (4\pi^2)^{\omega} }{\Gamma_g\big(\frac12+\omega\big)P_g(\omega)} 
        \\&  \times \int_{(A)}(4\pi^2)^{-s}P_g(s) \Gamma_g\Big(\frac12+s\Big)  \zeta_{p^\nu D} (1+2s)  (p^\nu D)^{s} 
    \frac{\xi\big(\frac12+s-\omega\big)^5}{\xi\big(\frac12\big)^5(s-\omega)}     
  \\&  \times
    \int_{(-\frac12-\delta)}   \frac{\Gamma\big(\frac{k-1+z}{2}\big)}{\Gamma\big(\frac{k+1-z}{2}\big)}
    \big(2\pi\sqrt{\ell}\big)^{-z}  \sum_{\substack{n\geq 1, \\ (n, p)=1}}
    \frac{\lambda_g(n)}{n^{\frac12+s+\frac{z}{2}}} \sum_{p^\nu \mid c} 
    \frac{S\big(n, \ell; c\big)}{c^{1-z}}
\mathop{dz} \mathop{ds} .
    \end{split}
    \end{equation}     
By applying \eqref{eq:n p rel prime} to  the summation over $n$ in the above, we separate it as follows.
    \begin{align*}
     \mathcal{S}_\ell(s, z)
    : =
    \sum_{\substack{n\geq 1\\ (n, p)=1}}
    \frac{\lambda_g(n)}{n^{\frac12+s+\frac{z}{2}}}  
    \sum_{p^\nu \mid c} 
    \frac{S\big(n, \ell; c\big)}{c^{1-z}} 
    =\,  \mathcal{S}_\ell(s, z; 0) - \lambda_g(p)  \mathcal{S}_\ell(s, z; 1)+  \mathcal{S}_\ell(s, z;2) ,
    \end{align*}
where we set
    \[
    \mathcal{S}_\ell(s, z; B) := \sum_{n\geq 1}
    \frac{\lambda_g(n)}{(np^B)^{\frac12+s+\frac{z}{2}}}  
    \sum_{p^\nu \mid c} 
    \frac{S\big(np^B, \ell; c\big)}{c^{1-z}}  \quad
    \text{for }\,\, B= 0, 1, 2.
    \]
Each of these three sums is easily seen to be absolutely convergent by using the Weil bound $S(a, b; c) \leq 2\sqrt{c} $. Thus, upon changing the order of summation, we can write 
    \[
     \mathcal{S}_\ell(s, z; B)
     =  \sum_{p^\nu \mid c} \frac{1}{c^{1-z}} \sum_{r \overline{r} \equiv 1 \text{(mod } c )} e \Big(\frac{2\pi i \ell r}{c}\Big)
    \sum_{n\geq 1} e \Big(\frac{2\pi i \overline{r} np^B}{c}\Big) \frac{\lambda_g(n)}{(np^B)^{\frac12+s+\frac{z}{2}}} 
    \quad 
    \text{for} \quad
    B=0, 1, 2.
    \]
Here we recall the definition of twisted $L$-functions. For $(a, c)=1$ and a newform $g$, we have
    \[
    L\Big(\tilde{s}, g, \frac{a}{c} \Big) = \sum_{n\geq 1} \frac{\lambda_g(n) e \big(\frac{an}{c}\big)}{n^{\tilde{s}}} \quad \text{for } \, \Re(\tilde{s}) >1. 
    \]
By letting $\tilde{s}=\frac12+s+\frac{z}{2}$ and choosing $\Re \big(\frac12+s+\frac{z}{2}\big)>1$, that is, $\delta <2 $, we further have
    \[
    \mathcal{S}_\ell(s, z; B)  =  \frac{1}{p^{B(\frac12+s+\frac{z}{2}})}  \sum_{p^\nu \mid c}  \frac{1}{c^{1-z}} \sum_{r \overline{r} \equiv 1 \text{(mod } c )} e \Big(\frac{2\pi i \ell r}{c}\Big)
    L\Big(\frac12+s+\frac{z}{2}, g, \frac{\overline{r} p^B}{c}\Big) \quad 
    \text{for } \,\,\,
    B=0, 1, 2.
    \]
We now move the line of integration in \eqref{eq:N ND 11} from $\Re z=\alpha= -\frac12-\delta$ to $\Re z = -\mathcal{A}$ with $\frac12 + \delta < \mathcal{A}=2A+1+2\delta < k-1$. By the functional equation of the twisted $L$-function, we have the following. For $(c, D)=1$, 
    \[
    L\Big(\tilde{s}, g, \frac{a}{c} \Big) 
    = 
     i^{k'}
    \Big(\frac{c\sqrt{D}}{2\pi} \Big)^{1-2\tilde{s}}
    \,
    \frac{\Gamma\big(1-\tilde{s}+\frac{k'-1}{2}\big)}{\Gamma\big(\tilde{s}+\frac{k'-1}{2}\big)}
    L\Big(1-\tilde{s}, g, \frac{-\overline{aD}}{c} \Big) .
    \]
For $D \mid c$, we have
    \[
    L\Big(\tilde{s}, g, \frac{a}{c} \Big) 
    = 
     i^{k'}
    \Big(\frac{c}{2\pi} \Big)^{1-2\tilde{s}}
    \,
    \frac{\Gamma\big(1-\tilde{s}+\frac{k'-1}{2}\big)}{\Gamma\big(\tilde{s}+\frac{k'-1}{2}\big)}
    L\Big(1-\tilde{s}, g, \frac{-\overline{a}}{c} \Big) .
    \]
Thus, it follows from the two forms of the functional equation that 
    \begin{equation}\label{eq:S ell s z}
    \mathcal{S}_\ell(s, z)
    =  i^{k'} (2\pi)^{2s+z}
    \frac{\Gamma\big(-s-\frac{z}{2}+\frac{k'}{2}\big)}{\Gamma\big(s+\frac{z}{2}+\frac{k'}{2}\big)}
    \sum_{B=0, 1, 2}\Big( \mathcal{S}'_\ell(s, z; B)+ {\mathcal{S}}^{''}_\ell(s, z; B) \Big),
    \end{equation}
where for each $B=0, 1, 2$ we set
    \begin{align*}
    \mathcal{S}'_\ell(s, z; B) &:= 
     \frac{ D^{-s-\frac{z}{2}}  }{p^{B(\frac12+s+\frac{z}{2}})} \sum_{\substack{p^\nu \mid c,\\ (c, D)=1}} \frac{1}{c^{1+2s}}  
    \sum_{n\geq 1} \frac{\lambda_g(n)}{n^{\frac12-s-\frac{z}{2}}}
    \sum_{r\overline{r} \equiv 1 \text{(mod } c )} e \Big(\frac{2\pi i \ell r}{c}\Big) e \Big(-\frac{2\pi i    nr \overline{p^B D} }{c}\Big) ,
    \\
    {\mathcal{S}}''_\ell(s, z; B) &:=    \frac{1}{p^{B(\frac12+s+\frac{z}{2}})}    \sum_{\substack{p^\nu \mid c,\\ D\mid c}} \frac{1}{c^{1+2s}} 
    \sum_{n\geq 1} \frac{\lambda_g(n)}{n^{\frac12-s-\frac{z}{2}}}
    \sum_{r\overline{r} \equiv 1 \text{(mod } c )} e \Big(\frac{2\pi i \ell r}{c}\Big) e \Big(-\frac{2\pi i   nr  \overline{p^B}}{c}\Big) .
    \end{align*}
By trivially bounding the sum over $r$ by $\phi(c)$, we can write the bound, for example, that
    \begin{align*} 
    \mathcal{S}'_\ell(s, z; 0) & \ll 
    D^{-A+\frac{\mathcal{A}}{2}} \sum_{n\geq 1} \frac{1}{n^{1+\delta-\epsilon}}
    \sum_{\substack{p^\nu \mid c,\\ (c, D)=1}} \frac{\phi(c)}{c^{1+2A}} 
    \ll D^{-A+\frac{\mathcal{A}}{2}}  \frac{\phi(p^\nu)}{p^{\nu(1+2A)}} \sum_{c\geq 1} \frac{\phi(c)}{c^{1+2A}} 
    \ll \frac{D^{-A+\frac{\mathcal{A}}{2}} }{p^{2A\nu}}.
    \end{align*}
The same bound holds for all other five terms, so
    \[
    \sum_{B=0, 1, 2}\Big( \mathcal{S}'_\ell(s, z; B) + {\mathcal{S}}^{''}_\ell(s, z; B) \Big)
    \ll \frac{D^{-A+\frac{\mathcal{A}}{2}} }{p^{2A\nu}} .
    \]
In \eqref{eq:N ND 11}, we let $z=-\mathcal{A}+it_1=-(2A+1+2\delta)+it_1$ and $s=A+it_2$. Then by the above bound, we have
    \begin{align*}
    &\frac{1}{2\pi i} \int_{(-\mathcal{A})} 
    \frac{\Gamma\big(\frac{k-1+z}{2}\big)}{\Gamma\Big(\frac{k+1-z}{2}\Big)} 
    (2\pi)^{-z} \ell^{-\frac{z}{2}}
    \mathcal{S}_\ell(s, z) \mathop{dz}
    \\
    &\ll
     \frac{\ell^{\frac{\mathcal{A}}{2}} D^{-A+\frac{\mathcal{A}}{2}} }{p^{2A\nu}}  
     \int_{-\infty}^\infty \big|k+it_1\big|^{-1-\mathcal{A}} \, \big| k' + i\big(2t_2+t_1 \big)\big|^{-2A+\mathcal{A}}  \mathop{dt_1},
    \end{align*}
where we applied Stirling's approximation to  bound the ratios of $\Gamma$-values. Then the above is
    \[
    \ll  \frac{D^{-A+\frac{\mathcal{A}}{2}} \ell^{\frac{\mathcal{A}}{2}}  }{p^{2A\nu} k^C}  |a+it_2|^{-2A+\mathcal{A}} 
    \quad \,\, \text{for } \,\, a:= \max\{k, k'\}.
    \]
Now, recall that we chose $\mathcal{A}-2A= 1+2\delta>1$. By using the above bound in \eqref{eq:N ND 11}, we obtain
    \begin{align*}
   &\mathcal{M}_{1,1}^{ND}(\omega)
    \ll    \frac{D^{-A+\frac{\mathcal{A}}{2}} \ell^{\frac{\mathcal{A}}{2}}  }{p^{2A\nu} k^{\mathcal{A}}}  
    \frac{(p^\nu D)^{-\Re(\omega)+A}}{\left|P_g(\omega)\Gamma_g\big(\frac12+\omega\big)\right|}
    \\
      & \times \int_{-\infty}^\infty 
     \bigg|
     \xi\Big(\frac12+A+it_2-\omega\Big)^5 
    P_g(A+it_2) \Gamma_g\Big(\tfrac12+A+it_2\Big) 
    \zeta_{p^\nu D}(1+2A+2it_2)
    \frac{ |a+it_2|^{-2A+\mathcal{A}} }{A+it_2-\omega} \bigg|
    \mathop{dt_2}.
    \end{align*}
Since we chose $1< A \leq \frac{k}{2}-1-\epsilon$, $\zeta_{p^\nu D} (1+2A+2it_2) \ll 1$. Applying the definition of $ \xi\big(\frac12+A+it_2-\omega\big)$ and bounding the $\Gamma$-factor using Stirling's formula in \eqref{eqn: stirling-formula} give
      \[
     \xi\Big(\frac12+A+it_2-\omega\Big)^5 
     \ll (A+|t_2|)^{10} \, |t_2-\Im(\omega)|^{-\frac54+\frac{5A}{2}-5\Re(\omega)} e^{-\frac{5\pi |t_2-\Im(\omega)|}{4}},
        \]
      for   $|t_2| \geq 1$ provided that $|\omega|$ is bounded.
      
      By combining these bounds, we obtain
 \begin{align*}
     \mathcal{M}_{1,1}^{ND}(\omega)
    \ll  &\,  \frac{D^{\frac{\mathcal{A}}{2}-\Re(\omega)} \ell^{\frac{\mathcal{A}}{2}}  }{p^{\nu(A+\Re(\omega))} k^{\mathcal{A}}}  \
    \\
    &\times \int_{-\infty}^\infty 
     \bigg| \frac{ P_g(A+it_2) \Gamma_g\big(\frac12+A+it_2\big) }{P_g(\omega)\Gamma_g\big(\frac12+\omega\big)} \bigg|
     \frac{ |a+it_2|^{-2A+\mathcal{A}}(A+|t_2|)^{\frac{31}{4}+\frac{5A}{2}-5\Re(\omega)} }
      {e^{\frac{5\pi |t_2-\Im(\omega)|}{4}}}     \mathop{dt_2}.
    \end{align*}
Since $P_g(0)=1$, it follows that $P_g(\omega)\gg 1$ provided that $|\omega|$ is sufficiently small. In fact, one could fix a choice of $P_g$ so that $P_g(\omega)\geq \frac12$. Since $P_g$ is a polynomial, we have 
     \[
    P_g(A+it_2) \ll (1+|t_2|)^{K+1} \quad \text{for some }\, K>0.
    \]
By definition, for $k>k'$
    \begin{align*}
    \frac{\Gamma_g\big(\frac12+A+it_2\big) }{\Gamma_g\big(\frac12+\omega\big) }
    = \frac{\Gamma\big(A+it_2+\frac{k-k'+1}{2}\big)
    \Gamma\big(A+it_2+\frac{k+k'-1}{2}\big) }{\Gamma\big(\frac{k-k'+1}{2}+\omega\big)
    \Gamma\big(\frac{k+k'-1}{2}+\omega\big) }.
    \end{align*}
By applying Lemma \ref{lem:log derivative gamma g bound} since $|A| < \frac{1+k-k'}{2}$ with $k>2k'-2$, we obtain
    \begin{align*}
    \frac{\Gamma_g\big(\frac12+A+it_2\big) }{\Gamma_g\big(\frac12\big) }
    \ll 
   k^{-2\Re(\omega)-k+1} \big(k^2+t_2^2\big)^{A+\frac{k-1}{2}}.
    \end{align*}
Hence,
\begin{align}\label{eqn:bound-1st-moment-nd}
     \mathcal{M}_{1,1}^{ND}(\omega)
   & \ll   \frac{D^{\frac{\mathcal{A}}{2}-\Re(\omega)} \ell^{\frac{\mathcal{A}}2}}{p^{\nu(A+\Re(\omega))} k^{2\Re(\omega)}}  
    \int_{-\infty}^\infty 
     \Big(1+\frac{t_2^2}{k^2}\Big)^{\frac{\mathcal{A}}{2}+\frac{k-1}{2}}
      (A+|t_2|)^{\frac{35}{4}+K+\frac{5A}{2}-5\Re(\omega)} e^{-\frac{5\pi |t_2-\Im(\omega)|}{4}} \mathop{dt_2}
    \nonumber  \\&\ll  \frac{D^{A+\frac12+\delta-\Re(\omega)} \ell^{A+\frac12+\delta} }{p^{\nu(A+\Re(\omega))} k^{2\Re(\omega)}} .
    \end{align}
One can similarly obtain the same bound for $ \mathcal{M}_{1,2}^{ND}(\omega)$, $ \mathcal{M}_{2,1}^{ND}(\omega)$ and $ \mathcal{M}_{2,2}^{ND}(\omega).$


\subsection{Proof of Theorem \ref{thm:first twisted moment}} Combining  \eqref{eqn:1st-moment-D} and \eqref{eqn:bound-1st-moment-nd}  gives
\begin{align*}
M_g(p^{\nu},\ell,\omega;1)
&=  \left(4\pi^2\right)^\omega \Big(1-\frac{1}{p}\Big)^2   \Big(1-\frac{1}{D}\Big) \frac{ \lambda_g(\ell)}{\sqrt\ell} 
       \frac{\Gamma_g\left(\frac12\right)P_g(0)\xi\left(\frac12+\omega\right)^5}{\Gamma_g\left(\frac12+\omega\right)P_g(\omega)\xi\big(\frac12\big)^5}
        \frac{(p^\nu D)^\omega-(p^\nu D)^{-\omega}}{2\omega} 
        \\&\quad+ \Big(1-\frac{1}{p}\Big) \frac{ \lambda_g(\ell)}{\sqrt\ell} \left\{ 
     \zeta_{p^\nu D} (1+2\omega)\ell^{-\omega}
  +  \frac{(4\pi^2)^{\omega} \Gamma_g(\frac12-\omega)}{(4\pi^2)^{-\omega} \Gamma_g(\frac12+\omega)}     \zeta_{p^\nu D} (1-2\omega)
   \ell^{\omega}\right\}
   \\&\quad
   +O\Big((1+|\Im(\omega)|)^B \ell^{-\frac14+\epsilon} (p^\nu D)^{\left(-\frac{5}{32}+\epsilon\right)}\Big)
   +O\bigg( \frac{D^{A+\frac12+\delta-\Re(\omega)} \ell^{A+\frac12+\delta}  }{p^{\nu(A+\Re(\omega))} k^{2\Re(\omega)}}\bigg),
\end{align*}
as desired.


\section{The Twisted Second Moment: Initial Setup}\label{sec:second moment}

Let $\ell$ be a positive integer such that $(\ell,p)=1$. Let $\omega=(\omega_1, \omega_2)$ be a pair of complex numbers satisfying \eqref{eqn:shift-condition}. For a fixed primitive form $g\in H_{k'}^*(1)$, we define 
    \[
    M_g(p^{\nu},\ell, \omega; 2)
    = \sum^h_{f\in H^*_k(p^\nu)} 
    \lambda_f(\ell)
    L\Big(\frac12+\omega_1, f\otimes g\Big)
    L\Big(\frac12+\omega_2, f\otimes g\Big).
    \]
We start by applying \eqref{eq:AF} to each $L$-value which gives
    \begin{equation}\label{eq:Mg setup}
    \begin{split} 
    M_g(p^{\nu},\ell, \omega; 2)&= p^{-\nu(\omega_1+\omega_2)} \Big(M(\omega_1,\omega_2)+ \ve_{\omega_1}(f\otimes g)  M(-\omega_1,\omega_2)\\&\hspace{2em}+ \ve_{\omega_2}(f\otimes g)M(\omega_1,-\omega_2)+ \ve_{\omega_1}(f\otimes g)\ve_{\omega_2}(f\otimes g)M(-\omega_1,-\omega_2)\Big),    \end{split}
    \end{equation}
where we set
\[
M(\omega_1,\omega_2):=\sum_{m, n \geq 1}
    \frac{\lambda_g(m) \lambda_g(n)}{\sqrt{mn}} 
    V_{g, \omega_1}\Big(\frac{m}{p^\nu }\Big)
    V_{g, \omega_2}\Big(\frac{n}{p^\nu }\Big)
     \sum^h_{f\in H^*_k(p^\nu)} \lambda_f(m) \lambda_f(n) \lambda_f(\ell).
  \]
  In view of \eqref{eq:Mg setup}  we let
   \[
   M_1:=p^{-\nu(\omega_1+\omega_2)} M(\omega_1,\omega_2),\quad M_2:=p^{-\nu(\omega_1+\omega_2)} M(-\omega_1,\omega_2),
   \] 
   \[
   M_3:=p^{-\nu(\omega_1+\omega_2)} M(\omega_1,-\omega_2), \,\, \text{ and} \quad M_4:=p^{-\nu(\omega_1+\omega_2)} M(-\omega_1,-\omega_2),
   \]
   so that 
   \begin{equation}\label{eqn:Mg setup with Mi}
    M_g(p^{\nu},\ell, \omega; 2)= M_1+ \ve_{\omega_1}(f\otimes g)  M_2+ \ve_{\omega_2}(f\otimes g)M_3+ \ve_{\omega_1}(f\otimes g)\ve_{\omega_2}(f\otimes g)M_4.
   \end{equation}


\subsection{Application of the Trace Formula} 

We shall focus on $M_1$ since $M_j$ for $j=2,3,4$ can be treated identically. By Lemma \ref{Hecke relation}, we have 
        \[
    M_1= 
     p^{-\nu(\omega_1+\omega_2)} 
     \sum_{d\mid \ell} 
     \sum_{m, n \geq 1}
    \frac{\lambda_g(m) \lambda_g(d n)}{\sqrt{mn d}} 
    V_{g, \omega_1}\Big(\frac{m}{p^\nu }\Big)
    V_{g, \omega_2}\Big(\frac{d n}{p^\nu }\Big)
     \sum^h_{f\in H^*_k(p^\nu)} 
     \lambda_f(m)\lambda_f\Big( \frac{n\ell}{d}\Big). 
    \]
Notice that the condition $(d, p)=1$ is dropped since $(\ell, p)=1$ as per our assumption. 

Applying Theorem \ref{Rouymi} separates $M_1$ into a diagonal part and two nondiagonal parts as follows. We have    \[
    M_1
    =
    M_1^{D}+
    M^{ND}_{1,1}+
    M^{ND}_{1,2},
    \]
where
    \[
    M^D_{1} 
    = \Big(1-\frac{1}{p}\Big) \frac{p^{-\nu(\omega_1+\omega_2)} }{\sqrt\ell} 
     \sum_{de= \ell } 
     \sum_{\substack{n\geq 1,\\ (n, p)=1}}
    \frac{\lambda_g (en) \lambda_g(dn)}{n}
    V_{g, \omega_1}\Big(\frac{en}{ p^\nu }\Big)
    V_{g, \omega_2}\Big(\frac{dn}{p^\nu }\Big),
    \]
      \begin{align*}
    M^{ND}_{1,1} 
    =&  \, 2\pi i^{-k} 
    p^{-\nu(\omega_1+\omega_2)} 
     \sum_{d\mid \ell}
     \frac{1}{\sqrt d}
     \sum_{\substack{m, n \geq 1,\\ (mn\ell/d, p)=1}}
    \frac{\lambda_g(m) \lambda_g(n d)}{\sqrt{mn}}
    V_{g, \omega_1}\Big(\frac{m}{p^\nu }\Big) 
    V_{g, \omega_2}\Big(\frac{nd}{p^\nu }\Big)
    \\
    &\times \sum_{p^\nu \mid c} 
    \frac{S\big(m, \frac{n\ell}{d}; c\big)}{c}
    J_{k-1}\Big(\frac{4\pi\sqrt{mn\ell}}{c\sqrt d}\Big),
    \end{align*}
and
    \begin{align*}
    M^{ND}_{1,2} 
    =&\, -\frac{2\pi i^{-k} 
    p^{-\nu(\omega_1+\omega_2)} }{p}  
     \sum_{d\mid \ell} 
     \frac{1}{\sqrt d}
     \sum_{\substack{m, n \geq 1,\\ (mn\ell/d, p)=1}}
    \frac{\lambda_g(m) \lambda_g(n d)}{\sqrt{mn}} 
    V_{g, \omega_1}\Big(\frac{m}{p^\nu }\Big) 
    V_{g, \omega_2}\Big(\frac{nd}{p^\nu }\Big)
    \\
    &\times \sum_{p^{\nu-1} \mid c} 
    \frac{S\big(m, \frac{n\ell}{d}; c\big)}{c}
    J_{k-1}\Big(\frac{4\pi\sqrt{mn\ell}}{c\sqrt d}\Big).
    \end{align*}


\subsection{Diagonal Terms} 

By using Lemma \ref{Hecke relation} in the form
    \[
    \lambda_g(nd)=\sum_{a \mid (d, n)} \mu(a) \lambda_g\Big(\frac{n}{a}\Big)\lambda_g\Big(\frac{d}{a}\Big),
    \]
we further write
    \begin{equation*}\label{eq:MD1}
    M^D_{1} 
    = \Big(1-\frac{1}{p}\Big) \frac{p^{-\nu(\omega_1+\omega_2)} }{\sqrt\ell} 
     \sum_{de=\ell} 
     \sum_{ab=d} \frac{ \mu(a) \lambda_g(b)}{a}
     \sum_{\substack{n\geq 1,\\ (n, p)=1}}
    \frac{\lambda_g(n) \lambda_g(aen) }{n} 
    V_{g, \omega_1}\Big(\frac{aen}{ p^\nu }\Big)
    V_{g, \omega_2}\Big(\frac{adn}{p^\nu }\Big) .
    \end{equation*}
 Using the same notation and following the same argument we get
   \begin{equation*}\label{eq:MD2}
    M^D_{2} 
    = \Big(1-\frac{1}{p}\Big) \frac{p^{-\nu(\omega_1+\omega_2)} }{\sqrt\ell} 
     \sum_{de=\ell} 
     \sum_{ab=d} \frac{ \mu(a) \lambda_g(b)}{a}
     \sum_{\substack{n\geq 1,\\ (n, p)=1}}
    \frac{\lambda_g(n) \lambda_g(aen) }{n} 
    V_{g, -\omega_1}\Big(\frac{aen}{ p^\nu }\Big)
    V_{g, \omega_2}\Big(\frac{adn}{p^\nu }\Big),
    \end{equation*}
    \begin{equation*}\label{eq:MD3}
    M^D_{3} 
    = \Big(1-\frac{1}{p}\Big) \frac{p^{-\nu(\omega_1+\omega_2)} }{\sqrt\ell} 
     \sum_{de=\ell} 
     \sum_{ab=d} \frac{ \mu(a) \lambda_g(b)}{a}
     \sum_{\substack{n\geq 1,\\ (n, p)=1}}
    \frac{\lambda_g(n) \lambda_g(aen) }{n} 
    V_{g, \omega_1}\Big(\frac{aen}{ p^\nu }\Big)
    V_{g, -\omega_2}\Big(\frac{adn}{p^\nu }\Big) ,
    \end{equation*}
and 
\begin{equation*}\label{eq:MD4}
    M^D_{4} 
    = \Big(1-\frac{1}{p}\Big) \frac{p^{-\nu(\omega_1+\omega_2)} }{\sqrt\ell} 
     \sum_{de=\ell} 
     \sum_{ab=d} \frac{ \mu(a) \lambda_g(b)}{a}
     \sum_{\substack{n\geq 1,\\ (n, p)=1}}
    \frac{\lambda_g(n) \lambda_g(aen) }{n} 
    V_{g, -\omega_1}\Big(\frac{aen}{ p^\nu }\Big)
    V_{g, -\omega_2}\Big(\frac{adn}{p^\nu }\Big) .
    \end{equation*}

We thus arrive at the diagonal contribution to $M_g(p^{\nu},\ell, \omega; 2)$ which we denote by $M^{D}$ and calculate as
\begin{equation*}\label{eqn:MD}
    M^{D}= M^{D}_g(p^{\nu},\ell, \omega; 2)=M^{D}_1+ \ve_{\omega_1}(f\otimes g)  M^{D}_2+ \ve_{\omega_2}(f\otimes g)M^{D}_3+ \ve_{\omega_1}(f\otimes g)\ve_{\omega_2}(f\otimes g)M^{D}_4.
   \end{equation*}

    
\subsection{Nondiagonal Terms} 

In what follows we focus on $M^{ND}_{1, 1} $ and $M^{ND}_{1, 2}$ since $M^{ND}_{j,1}$ and $M^{ND}_{j,2}$ for $j=2,3,4$ can be treated similarly. By using Lemma \ref{Hecke relation} and again substituting $de=\ell$, we obtain 
    \begin{align*}
    M^{ND}_{1, 1} 
    =&\,  2\pi i^{-k} 
    p^{-\nu(\omega_1+\omega_2)} 
     \sum_{de= \ell} 
     \frac{1}{\sqrt d}
     \sum_{\substack{ab=d, \\ (a, D)=1}}
     \frac{\mu(a) \lambda_g(b)}{\sqrt a}
     \sum_{\substack{m, n \geq 1,\\ (mn, p)=1}}
    \frac{\lambda_g(m) \lambda_g(n)}{\sqrt{mn}}
     \\
    &\times 
    V_{g, \omega_1}\Big(\frac{m}{p^\nu }\Big) 
    V_{g, \omega_2}\Big(\frac{adn}{p^\nu }\Big)
    \sum_{p^\nu \mid c} 
    \frac{S\big(m, aen; c\big)}{c}
    J_{k-1}\Big(\frac{4\pi\sqrt{aemn}}{c}\Big) .
    \end{align*}
    
    At this stage, we apply a dyadic partitioning in the $m$ and $n$ summations using  smooth functions $\eta_M(x)$ that are compactly supported in $[M/2, 2M]$ for some $M$ such that
\[
x^j \eta_M^{(j)}(x) \ll_i 1 \quad \text{for any } j \geq 0,
\]
and define
\[
\eta(x) := \sum_{M \geq 1} \eta_M(x)=\begin{cases}0& x \leq \tfrac{1}{2}\\ 1& x\geq1.\end{cases}
\]
We also require that
\[
\sum_{M \leq X} 1 \ll \log X.
\]
  We define the functions 
    \[
    F_{M, N}(m, n) := \frac{1}{\sqrt{mn}}  \eta_M(m) \eta_N(n)
    V_{g, \omega_1}\Big(\frac{m}{p^\nu }\Big) 
    V_{g, \omega_2}\Big(\frac{adn}{p^\nu }\Big)
    \quad
    \text{and} 
    \quad
    F(m, n) := \sum_{M, N} F_{M, N}(m, n) .
    \]
This allows us to write
    \begin{equation}\label{eq:M ND 11} 
    \begin{split}
    M^{ND}_{1, 1}    
    = &\,  \frac{2\pi i^{-k} }{
    p^{\nu(\omega_1+\omega_2)}} 
     \sum_{de= \ell} 
     \frac{1}{\sqrt d}
     \sum_{ab=d}
     \frac{\mu(a) \lambda_g(b)}{\sqrt a}
     \sum_{\substack{m, n \geq 1,\\ (mn, p)=1}}
    \lambda_g(m) \lambda_g(n)
    \\
    & \times \sum_{M, N} F_{M, N}(m, n)
    \sum_{p^\nu \mid c} 
    \frac{S\big(m, aen; c\big)}{c}
    J_{k-1}\Big(\frac{4\pi\sqrt{aemn}}{c}\Big) .
    \end{split}
    \end{equation}
 We require the following bound for the function $F_{M,N}(x,y)$.
 \begin{lem}[(7.6) in \cite{KMV RS}] 
\label{lem:F bound} 
Suppose $\omega_1,\omega_2\in\mathbb{C}$ satisfy \eqref{eqn:shift-condition}. In particular, $|\omega_1|,|\omega_2|\leq\eta$ for some small $\eta>0$. Define
    \[
    F_{M, N}(x, y) 
    := \frac{1}{\sqrt{xy}}  \eta_M(x) \eta_N(y)
    V_{g, \omega_1}\Big(\frac{x}{p^\nu}\Big) 
    V_{g, \omega_2}\Big(\frac{y}{p^\nu}\Big).
    \]
For all $i, j, A, A' \geq 0$, we have
    \[
    x^i y^j \frac{\partial^i}{\partial^i x} 
    \frac{\partial^j}{\partial^j y}
    F_{M, N}(x, y)
    \ll_{\eta}(MN)^{-\frac12} 
    (\log(p^\nu ))^{i+j}
    \Big(\frac{p^\nu}{x}\Big)^A
    \Big(\frac{p^\nu}{ady}\Big)^{A'}.
    \]
\end{lem}

To shorten the expression for $M^{ND}_{1, 1}$ even further, we set
    \begin{equation}\label{eq:T M N c}
    T_{M, N}(c):=
    c \sum_{\substack{m, n \geq 1,\\ (mn, p)=1}}
    \lambda_g(m) \lambda_g(n)
    F_{M, N}(m, n) S\big(m, aen ; c\big)
    J_{k-1}\Big(\frac{4\pi\sqrt{aemn}}{c}\Big)  
    \end{equation}   
    and
    \[
    T(c):= \sum_{M, N \geq 1} T_{M, N}(c).
    \]
Using this notation in \eqref{eq:M ND 11} gives
    \begin{equation}\label{eq:MND in Tc}
    M^{ND}_{1, 1} 
    = \frac{2\pi i^{-k}}{ p^{\nu(\omega_1+\omega_2)} }
     \sum_{d\mid \ell} \frac{1}{\sqrt d} 
      \sum_{ab=d}
     \frac{\mu(a) \lambda_g(b)}{\sqrt a}
     \sum_{p^\nu \mid c} \frac{T(c)}{c^2}.
     \end{equation}
    Similarly, we can write
    \begin{equation}\label{eq:MND in Tc 12}
    M^{ND}_{1, 2} 
    = -\frac{2\pi i^{-k}}{p^{1+\nu(\omega_1+\omega_2)}}  
     \sum_{d\mid \ell} \frac{1}{\sqrt d} 
      \sum_{ab=d}
     \frac{\mu(a) \lambda_g(b)}{\sqrt a} \sum_{p^{\nu-1} \mid c} \frac{T(c)}{c^2} .
    \end{equation}


Observe that the sum over $m, n$ in \eqref{eq:T M N c} has the restriction $(mn, p)=1$. We handle this condition by an inclusion-exclusion argument and obtain three sums over $m$ and $n$ as follows. We have
    \begin{align*}
  T_{M, N}(c)    &=\bigg\{ \sum_{m, n \geq 1}-
    \sum_{\substack{m, n \geq 1,\\ p \mid n}} 
    - \sum_{\substack{m, n \geq 1, \\ p\mid m, p\nmid n}} \bigg\}
    \lambda_g(m) \lambda_g(n)
    F_{M, N}(m, n) S\big(m, aen; c\big)
    J_{k-1}\Big(\frac{4\pi\sqrt{aemn}}{c}\Big).
    \end{align*}
Here, the sum over $m,n\geq 1, p\mid m, p \nmid n$ is $0$ by Lemma \ref{Kloosterman vanishes}. To the sum over $m,n\geq 1$ with $ p \mid n$, we apply \eqref{eq:n p rel prime} to get
    \begin{equation*}\label{eq:TMN as sum}
    \begin{split}
    T_{M, N}(c) 
    &=  c \sum_{m, n \geq 1}
    \lambda_g(m) \lambda_g(n)
    F_{M, N}(m, n) S\big(m, aen; c\big)
    J_{k-1}\Big(\frac{4\pi\sqrt{aemn}}{c}\Big) \\
    &-c\, \lambda_g(p)   \sum_{m, n \geq 1}
    \lambda_g(m) \lambda_g(n)
    F_{M, N}(m, np) S\big(m, aenp; c\big)
    J_{k-1}\Big(\frac{4\pi\sqrt{aemnp}}{c}\Big) \\
    &+c \sum_{m, n \geq 1}
    \lambda_g(m) \lambda_g(n)
    F_{M, N}(m, np^2) S\big(m, aenp^2; c\big)
    J_{k-1}\Big(\frac{4\pi\sqrt{aemnp^2}}{c}\Big). 
       \end{split}
    \end{equation*}
We set \begin{equation}\label{eqn:TScB}T_{M,N}(c,B)=c\sum_{m, n \geq 1}
    \lambda_g(m) \lambda_g(n)
    F_{M, N}(m, np^B) S\big(m, aenp^B; c\big)
    J_{k-1}\Big(\frac{4\pi\sqrt{aemnp^B}}{c}\Big)
    \end{equation} so that 
\begin{equation}\label{eqn:TMNc-TS}T_{M, N}(c) =T_{M,N}(c, 0)- \lambda_g(p)
    T_{M,N}(c, 1) +
     T_{M,N}(c, 2).
     \end{equation}


\section{Estimation of the Nondiagonal Contribution}\label{sec:second moment-1}

In this section we will  estimate the nondiagonal terms $M^{ND}_{j, 1}$ and $M^{ND}_{j, 2}$ while focussing on the case $j=1$ for simplicity.


\subsection{Restricting the Range of $M, N$ and $c$}

We will truncate the sums over $M, N$ and $c$ in $M^{ND}_{1, 1}$ and $M^{ND}_{1, 2}$ up to an error term that will later be ensured to be small.

Our first lemma shows that the contribution from large $M$ or large $N$ to the nondiagonal terms will be small with respect to the size of $p^\nu$.


\begin{lem}\label{restrict M N}
For $\delta=0$ or $1$, we have 
    \[
    \sum_{\max\{M, N\} \gg p^{\nu(1+\epsilon)} }
    \sum_{p^{\nu-\delta} \mid c} \frac{T_{M, N}(c) }{c^2} 
    \ll_{\epsilon, \mathcal{C}, v} (ae)^{\frac14+\epsilon}  p^{-\nu \mathcal{C}}
    \]
for any positive constant $\mathcal{C}$. 
\end{lem}


\begin{proof}
We recall first that $T_{M, N}(c) $ is defined in \eqref{eq:T M N c}.
Without loss of generality, we assume that 
$M\gg p^{\nu(1+\epsilon)}$ and $N\ll p^{\nu(1+\epsilon)}$. The other cases where $M\gg p^{\nu(1+\epsilon)}$ and  $N\gg p^{\nu(1+\epsilon)}$ or $M\ll p^{\nu(1+\epsilon)}$ and $N\gg p^{\nu(1+\epsilon)}$ can be studied similarly. 

We divide the sum over $c$ into $c \gg \sqrt{aemn}$ and $c \ll \sqrt{aemn}$ and bound the $J_{k-1}$-value accordingly. 
Using Weil's bound \eqref{eqn:weil-bound} and applying Lemma \ref{lem:Bessel} gives
    \begin{align*}
    \sum_{\substack{p^{\nu-\delta} \mid c, \\
    c \ll \sqrt{aemn}}} \frac{S\big(m, aen; c\big)}{c} 
    J_{k-1}\Big(\frac{4\pi\sqrt{aemn}}{c}\Big) 
    &\ll 
   \sum_{\substack{p^{\nu-\delta} \mid c, \\
    c \ll \sqrt{aemn}}}
    \frac{(m, aen, c)^{\frac12} \tau(c) }{\sqrt{c}}  
    \Big(\frac{4\pi\sqrt{aemn}}{c}\Big)^{-\frac12}
   \\& \ll
     (aemn)^{-\frac14+\frac{\epsilon}{2}}  
    \sum_{\substack{p^{\nu-\delta} \mid c, \\
    c \ll \sqrt{aemn}}}
    (m, aen, c)^{\frac12} \\&
    \ll(aemn)^{\frac14+\epsilon}p^{-\frac{1}{2}(\nu-\delta)}. 
    \end{align*}
For the range $c \gg \sqrt{aemn}$, we decompose the $c$-sum into dyadic intervals and apply  \eqref{eqn:weil-bound} and Lemma \ref{lem:Bessel} to get    \begin{align*}
    \sum_{\substack{p^{\nu-\delta} \mid c, \\
    c \gg \sqrt{aemn}}} \frac{S\big(m, aen; c\big)}{c} 
    J_{k-1}\Big(\frac{4\pi\sqrt{aemn}}{c}\Big) 
    & \ll
    \sum_{\substack{p^{\nu-\delta} \mid c, \\
    c \gg \sqrt{aemn}}} \frac{(m, aen, c)^{\frac12} \tau(c) }{\sqrt c} 
    \Big(\frac{4\pi\sqrt{aemn}}{c}\Big)^{k-1}
  \\&  \ll
    (aemn)^{\frac14+\epsilon}p^{-\frac{1}{2}(\nu-\delta)}
    \end{align*}
Hence,
    \[
    \sum_{p^{\nu-\delta} \mid c} \frac{S\big(m, aen; c\big)}{c} 
    J_{k-1}\Big(\frac{4\pi\sqrt{aemn}}{c}\Big) 
    \ll  (aemn)^{\frac14+\epsilon}p^{-\frac{1}{2}(\nu-\delta)}.
    \]
Then for any $\epsilon >0$, we have
    \begin{align*}
        \sum_{\substack{M \gg p^{\nu(1+\epsilon)},\\
     N\ll p^{\nu(1+\epsilon)} }}
    \sum_{p^{\nu-\delta} \mid c} \frac{T_{M, N}(c) }{c^2} 
   & \ll  (ae)^{\frac14+\epsilon}p^{-\frac{1}{2}(\nu-\delta)}
    \sum_{\substack{M \gg p^{\nu(1+\epsilon)},\\
     N\ll p^{\nu(1+\epsilon)} }}
    \sum_{\substack{m, n \geq 1, \\
    (mn, p)=1}}
    |\lambda_g(m) \lambda_g(n)
     F_{M, N}(m, n)| (mn)^{\frac14+\epsilon}
  \\ & \ll_{\epsilon,\eta}  \frac{ (ae)^{\frac14+\epsilon}
    p^{\nu(A+A')-\frac{\nu}{2}+\frac{\delta}{2}} }{(ad)^{A'}}
    \sum_{\substack{M \gg p^{\nu(1+\epsilon)},\\
     N\ll p^{\nu(1+\epsilon)} }}
    (MN)^{-\frac12} 
     \sum_{\substack{m \asymp M\\ n\asymp N}}
    m^{\frac14-A+\epsilon}
    n^{\frac14-A'+\epsilon},
    \end{align*}
  where we used Lemma \ref{lem:F bound} and Deligne's bound $|\lambda_g(m)| \ll_\epsilon m^\epsilon $ for the last inequality.
  Hence,  
  \begin{align*}
    \sum_{\substack{M \gg p^{\nu(1+\epsilon)},\\
     N\ll p^{\nu(1+\epsilon)} }}
    \sum_{p^{\nu-\delta} \mid c} \frac{T_{M, N}(c) }{c^2}  &\ll
     (ae)^{\frac14+\epsilon}
    p^{\nu(A+A')-\frac{\nu}{2}+\frac{\delta}{2}}
    \sum_{\substack{M \gg p^{\nu(1+\epsilon)},\\
     N\ll p^{\nu(1+\epsilon)} }}
    (MN)^{-\frac12} 
    M^{\frac54+\epsilon-A}
    N^{\frac54+\epsilon-A'}\\
    &\ll 
      (ae)^{\frac14+\epsilon}p^{\frac{\delta}{2}}
    p^{\nu(A'+\frac54-\epsilon A+O(\epsilon))}\sum_{N\ll p^{\nu(1+\epsilon)}}N^{\frac34-A'+\epsilon}
  \\&  \ll
    (ae)^{\frac14+\epsilon}p^{\frac{\delta}{2}}
p^{-\nu \mathcal{C}},
    \end{align*}
where in the last estimate we chose 
$A> \frac{7}{4}+\epsilon$, $A'> \frac{7}{4}+\epsilon$,
and $A>\frac{A'+\frac{5}{4}+\frac{11}{4}\epsilon+\epsilon^2+\mathcal{C}}{\epsilon}.
$
\end{proof}


Next we focus next on estimating the contribution from values of $c$ that are large in comparison to $M$ and $N$.


\begin{lem}\label{restrict C}
 Let $\delta$ be $0$ or $1$. Assume that $M, N \ll p^{\nu(1+\epsilon)}$. For $C \gg \sqrt{aeMN}$ we have 
      \[
    \sum_{\substack{c\geq C,\\ 
    p^{\nu-\delta} \mid c}} \frac{T_{M, N}(c)}{c^2}
    \ll_{\epsilon, \eta} 
     (ae)^\theta  \left(Cp^\delta\right)^{\epsilon}
    \bigg(\frac{\sqrt{aeMN}}{C}\bigg)^{k-1-2\theta},
    \]
where
\begin{equation}
\theta = \sqrt{\max(0, 1/4 - \lambda_1)} 
\end{equation}
and $\lambda_1$ is the smallest positive eigenvalue for the Hecke congruence subgroup $\Gamma_0(p^{\nu})$. Currently the best known estimate on $\lambda_1$ is due to Kim and Sarnak~\cite{Kim Sarnak}. Accordingly, we can take $\theta = \frac{7}{64}$.

\end{lem}

\begin{proof}
We consider 
    \[
    \sum_{\substack{ m\asymp M, n\asymp N, \\
    (mn, p)=1}}
    \lambda_g(m) \lambda_g(n)
    F_{M, N}(m, n)
    \sum_{\substack{p^{\nu-\delta} \mid c,\\ c\geq C}}
    \frac{S\big(m, aen; c\big)}{c} 
    J_{k-1}\Big(\frac{4\pi\sqrt{aemn}}{c}\Big).
    \]
 We decompose the sum over $c$ into dyadic intervals
    \[
    \sum_{\substack{p^{\nu-\delta} \mid c,\\ c\geq C}}
    = \sum_{j=0}^\infty \sum_{\substack{p^{\nu-\delta} \mid c,\\ c\in [2^jC, 2^{j+1}C]}}.
    \]
Upon the change of variables $c'=\frac{c}{p^{\nu-\delta}}$, $C'=\frac{C}{p^{\nu-\delta}}$, we write the inner sum corresponding to $j=0$  as
    \[
    \mathcal{T}_{j=0}
    =
    p^{\delta-\nu}
    \sum_{\substack{m, n \geq 1, \\
    (mn, p)=1}}
    \lambda_g(m) \lambda_g(n)
    F_{M, N}(m, n)
    \sum_{c' \in [C', 2C']}
    \frac{S\big(aen, m; c' p^{\nu-\delta}\big)}{c'} 
    J_{k-1}\Big(\frac{4\pi\sqrt{aemn}}{c' p^{\nu-\delta}}\Big). 
    \]
Now, we apply Theorem \ref{Matomaki} with
    \[
    H(m, n, c')
    = \frac{XF_{M, N}(m, n)}{c'}
    J_{k-1}\Big(\frac{4\pi\sqrt{aemn}}{c' p^{\nu-\delta}}\Big),
    \]
where
    \[
    X=  C'  \sqrt{MN} \bigg(\frac{C' p^{\nu-\delta}}{\sqrt{aeMN}}\bigg)^{k-1} 
    \Big(\frac{M}{p^\nu}\Big)^A
    \Big(\frac{N}{p^\nu}\Big)^{A'}.
    \]
This yields
    \begin{align*}
    \mathcal{T}_{j=0}
    &=  \frac{p^{\delta-\nu}}{X}
    \sum_{\substack{m, n \geq 1, \\
    (mn, p)=1}}
    \lambda_g(m) \lambda_g(n)
    \sum_{c' \in [C', 2C']}
    S\big(aen, m; c' p^{\nu-\delta}\big)  H(m, n, c')\\&
    \ll
     \frac{{C'}^{1+\epsilon} (ae)^{\theta}}{X} 
    \frac{\big(1+{X_{ae}}^{-1}\big)^{2\theta}}{1+X_{ae}}
    \Big(1+X_{ae}+\frac{\sqrt{M}}{\sqrt{p^{\nu-\delta}}}\Big)
    \Big(1+X_{ae}+\frac{\sqrt{N}}{\sqrt{p^{\nu-\delta}}}\Big)
    \\
    &\hspace{2em} \times \Big( \sum_{m\leq M} \lambda_g(m)^2\Big)^{1/2}
    \Big( \sum_{n\leq N} \lambda_g(n)^2\Big)^{1/2},
    \end{align*}
    where $X_{ae}=\frac{\sqrt{aeMN}}{p^{\nu-\delta}C'}$. By using the estimate $ \sum_{n\leq x} \lambda_g(n)^2 \ll x$, substituting the value of $X$, and using the condition $C>\sqrt{aeMN}$, we find that 
    \begin{align*}
    \mathcal{T}_{j=0}
         & \ll_{\eta} 
   {C'}^{\epsilon}  (ae)^{\theta}
     \bigg(\frac{\sqrt{aeMN}}{C}\bigg)^{k-1-2\theta} 
     \Big(\frac{p^\nu}{M} \Big)^A
    \Big(\frac{p^\nu }{N}\Big)^{A'}
    \Big(1+\frac{\sqrt M}{\sqrt{p^{\nu-\delta}}}\Big)
    \Big(1+\frac{\sqrt N}{\sqrt{p^{\nu-\delta}}}\Big).
    \end{align*}
Thus, we proved that
    \begin{align*}
        \mathcal{T}_{j=0}
    \ll_{\eta} 
    (ae)^{\theta}
    p^{\nu\epsilon}
    \left(\frac{C}{p^{\nu-\delta}}\right)^\epsilon
        \bigg(\frac{\sqrt{aeMN}}{C}\bigg)^{k-1-2\theta}. 
    \end{align*}
By repeating the same argument for $c$ with $p^{\nu-\delta}\mid c$ in other dyadic intervals $[2^j C, 2^{j+1} C]$ we complete the proof. 
\end{proof}


\subsection{Off-diagonal and Off-off-diagonal Terms}

In view of \eqref{eqn:TMNc-TS}, we will apply Vorono\"{\i} summation formula to $T_{M,N}(c, B)$ given in \eqref{eqn:TScB} for each $B=0, 1,2$. This will allow us to split the sum of nondiagonal terms  into the sum of off-diagonal and off-off-diagonal terms. 

Applying \eqref{eqn:kloosterman} in \eqref{eqn:TScB} yields
    \[
    T_{M,N}(c, B) 
    = \sum_{ n \geq 1}  \lambda_g(n)
    \sum_{\substack{f(\text{mod } c),\\ (f, c)=1}}
    e\Big(\frac{aefnp^B}{c}\Big)
    \, c\sum_{m \geq 1}
    \lambda_g(m) e\Big(\frac{m\overline{f}}{c}\Big)\tilde{F}_{M, N}(m, np^B),
    \]
where we set
    \[
    \tilde{F}_{M, N}(m, np^B)=
    F_{M, N}(m, np^B) 
    J_{k-1}\Big(\frac{4\pi\sqrt{aemnp^B}}{c}\Big).
    \]
Now, we apply Theorem \ref{Voronoi} to get
    \begin{align*}
    &T_{M,N}(c, B) 
    =  \phi(c)     \sum_{ n \geq 1}  \lambda_g(n) \lambda_g(aen p^B)
    \int_0^\infty \tilde{F}_{M, N}(x, np^B)
    J_g\Big(\tfrac{4\pi\sqrt{aenp^B x}}{c} \Big) \mathop{dx} \\
    &+\sum_{\substack{n\geq 1, \\ h\neq 0}}  
    \lambda_g(n) \lambda_g(aenp^B-h)
    \sum_{\substack{f(\text{mod } c),\\ (f, c)=1}}
    e\Big(\frac{fh}{c}\Big)
    \int_0^\infty \tilde{F}_{M, N}(x, np^B ) 
    J_g\Big(\tfrac{4\pi\sqrt{(aenp^B-h)x}}{c} \Big) \mathop{dx} .
    \end{align*}
For an integer $h$, we set
    \begin{equation*}\label{eqn:Th-def}
    T_{h}(c, B)
    =  \sum_{n\geq 1}  
    \lambda_g(n) \lambda_g(aenp^B - h)
    G_{M,N}\big(aenp^B-h, np^B\big) ,
    \end{equation*}
    where  \begin{equation}\label{eqn:G-def}
    G_{M,N}(z, y)= \int_0^\infty \tilde{F}_{M, N}(x, y) 
    J_g\Big(\frac{4\pi\sqrt{zx}}{c} \Big) \mathop{dx}.
    \end{equation}
With this notation, we can write
    \begin{equation}\label{eq:TS-cB}
    T_{M,N}(c, B)
    =\phi(c) 
    T_0(c,B)
    +  \sum_{h\neq 0} S(0, h ;c) 
    T_{h}(c, B) .
    \end{equation}
Applying Lemma \ref{restrict M N} in \eqref{eq:MND in Tc} and \eqref{eq:MND in Tc 12} we obtain    \begin{align*}
    M^{ND}_{1, 1}
    = &\, 2\pi i^{-k} p^{-\nu (\omega_1+\omega_2)}
    \sum_{de=\ell} \frac{1}{\sqrt d}
    \sum_{ab=d} \frac{\mu(a) \lambda_g(b)}{\sqrt a} \\
    &\times 
    \sum_{M, N \ll p^{\nu(1+\epsilon)}}\sum_{p^\nu \mid c} \frac{1}{c^2}\Big(T_{M,N}(c, 0) - \lambda_g(p)T_{M,N}(c, 1)+T_{M,N}(c, 2) \Big) 
    +O_{\epsilon, v, \mathcal{C}}\big(  \ell^{\frac{1}{4}+\epsilon} p^{-\nu \mathcal{C}} \big), 
    \\
    M^{ND}_{1, 2}
    = &\, -\frac{2\pi i^{-k}}{p} p^{-\nu (\omega_1+\omega_2)}
    \sum_{de=\ell} \frac{1}{\sqrt d}
    \sum_{ab=d} \frac{\mu(a) \lambda_g(b)}{\sqrt a} \\
    &\times 
    \sum_{M, N \ll p^{\nu(1+\epsilon)}}\sum_{p^{\nu-1} \mid c} \frac{1}{c^2}\Big(T_{M,N}(c, 0) - \lambda_g(p)T_{M,N}(c, 1)+T_{M,N}(c, 2) \Big)
    +O_{\epsilon, v, \mathcal{C}}\big( \ell^{\frac{1}{4}+\epsilon} p^{-\nu \mathcal{C}} \big). 
    \end{align*}
Further, by Lemma \ref{restrict C}, we have    \begin{align}\label{MND11-after-Vor}
     M^{ND}_{1, 1}
    &=  \frac{2\pi i^{-k}}{ p^{\nu (\omega_1+\omega_2)}}
    \sum_{\substack{ab=d,\\ de=\ell}} \frac{\mu(a) \lambda_g(b)}{\sqrt{ad}} 
    \sum_{M, N \ll p^{\nu(1+\epsilon)}} \sum_{\substack{p^\nu \mid c,\\ c\leq C }} \frac{TS^-(c, 0) - \lambda_g(p)TS^-(c, 1)+TS^-(c, 2)}{c^2} \nonumber\\
    &+O_{\epsilon, v}\bigg( { \sum_{\substack{ab=d,  \\ de=\ell}} \frac{(ae)^{\theta}b^{\epsilon}}{\sqrt{ad}}}  \sum_{M, N \ll p^{\nu(1+\epsilon)}}C^{\epsilon}\Big(\frac{\sqrt{aeMN}}{C}\Big)^{k-1-2\theta} \bigg), 
    \end{align}
   \begin{align} \label{MND12-after-Vor} M^{ND}_{1, 2}
    &=  \tfrac{-2\pi i^{-k}}{p^{1+\nu (\omega_1+\omega_2)}}
    \sum_{\substack{ab=d,  \\ de=\ell}} \tfrac{\mu(a) \lambda_g(b)}{\sqrt{ad}} 
    \sum_{M, N \ll p^{\nu(1+\epsilon)}}\sum_{\substack{p^{\nu-1} \mid c,\\ c\leq C}} \frac{TS^-(c, 0) - \lambda_g(p)TS^-(c, 1)+TS^-(c, 2)}{c^2} \nonumber\\
    &+O_{\epsilon, v}\bigg( p^{\epsilon}{ \sum_{\substack{ab=d,  \\ de=\ell}} \frac{(ae)^{\theta}b^{\epsilon}}{\sqrt{ad}}} \sum_{M, N \ll p^{\nu(1+\epsilon)}}C^{\epsilon}\Big(\frac{\sqrt{aeMN}}{C}\Big)^{k-1-2\theta} \bigg).
    \end{align}
The error terms in \eqref{MND11-after-Vor} and \eqref{MND12-after-Vor} will be estimated in Section \ref{sec:ood} once we choose $C$ optimally in terms of $p^{\nu}$, $M$, and $N$. 

Note that on the right hand side of \eqref{eq:TS-cB}, the first term  will be referred to as the off-diagonal part and the remaining part will be referred to as the off-off-diagonal part. The contribution of the off-diagonal part of $T_{M,N}(c, B)$ to $M^{ND}_{1, 1}$ will be denoted by $M_{1, 1}^{OD}(B)$, while that of its off-off-diagonal part will be denoted by $M_{1, 1}^{OOD}(B)$. 
More precisely, 
for each $B=0, 1, 2$, we define
     \begin{equation}\label{eq:M 11 OD}
    M^{OD}_{1, 1}(B)
    =  \frac{2\pi i^{-k} }{p^{\nu(\omega_1+\omega_2)}}
    \sum_{\substack{ab=d, \\ de=\ell}} \frac{\mu(a) \lambda_g(b)}{\sqrt {ad}} 
    \sum_{ M, N \ll p^{\nu(1+\epsilon)}} \sum_{\substack{c\leq C, \\ p^\nu  \mid c}} \frac{\phi(c)}{c^2}
    T_0(c,B) 
     \end{equation}
     and  
  \begin{equation}\label{eq:M 11 OOD}
    M^{OOD}_{1, 1}(B)
    =  \frac{2\pi i^{-k} }{p^{\nu(\omega_1+\omega_2)}}
    \sum_{\substack{ab=d, \\ de=\ell}} \frac{\mu(a) \lambda_g(b)}{\sqrt {ad}} 
   \sum_{h\neq0}  
    \sum_{ M, N \ll p^{\nu(1+\epsilon)}}\sum_{\substack{c\leq C, \\ p^\nu  \mid c}} \frac{S(0, h ;c) 
    }{c^2}
     T_{h}(c, B).
     \end{equation}
     We also set
    \begin{equation}\label{eqn:M11OD}
    M^{OD}_{1, 1}
    = M^{OD}_{1, 1}(0)-\lambda_g(p) M^{OD}_{1, 1}(1)+M^{OD}_{1, 1}(2) 
    \end{equation}
    and 
    \begin{equation}\label{eqn:M11OOD}
    M^{OOD}_{1, 1}
    = M^{OOD}_{1, 1}(0)-\lambda_g(p) M^{OOD}_{1, 1}(1)+M^{OOD}_{1, 1}(2).
\end{equation}
The terms $M^{OD}_{1, 2}(B)$, $M^{OOD}_{1, 2}(B)$ and $M^{OD}_{1,2}$ are defined similarly except for the extra factor of $-\frac{1}{p}$ and the condition in the $c$-sum being $p^{\nu-1}\mid c$.

 With this notation, we can write
    \[
    M^{ND}_{1, 1}
    =M^{OD}_{1, 1}+M^{OOD}_{1, 1}\]     and
\[
    M^{ND}_{1, 2}
    =M^{OD}_{1, 2}+M^{OOD}_{1, 2},
    \]
    up to the error terms indicated in \eqref{MND11-after-Vor} and \eqref{MND12-after-Vor}.
    
  In view of \eqref{eqn:Mg setup with Mi}, we denote the off-diagonal contribution to the second mixed moment by
    \begin{align}\label{eqn:off-diag-contr}
     M^{OD}
    &= M^{OD}_g(p^{\nu}, \ell, \omega; 2)\nonumber
     \\&= \,  M^{OD}_{1}
     +\ve_{\omega_1}(f\otimes g)M^{OD}_{2}+\ve_{\omega_2}(f\otimes g)M^{OD}_{3}
      +\ve_{\omega_1}(f\otimes g)\ve_{\omega_2}(f\otimes g) M^{OD}_{4} ,
    \end{align}
    and the off-off-diagonal contribution by 
    \begin{align}\label{eqn:off-off-diag-contr}
     M^{OOD}
    &= M^{OOD}_g(p^{\nu}, \ell, \omega; 2)\nonumber
     \\&= \,  M^{OOD}_{1}
     +\ve_{\omega_1}(f\otimes g)M^{OOD}_{2}+\ve_{\omega_2}(f\otimes g)M^{OOD}_{3}
      +\ve_{\omega_1}(f\otimes g)\ve_{\omega_2}(f\otimes g) M^{OOD}_{4}.
    \end{align}
Here we use the notation \begin{equation*}\label{eqn:MjOD}M^{OD}_{j}=M^{OD}_{j, 1}+ M^{OD}_{j,2}\end{equation*} and \begin{equation*}\label{eqn:MjOOD}M^{OOD}_{j}=M^{OOD}_{j, 1}+ M^{OOD}_{j,2}\end{equation*} 
 for $j=1,\cdots,4$.

 
\subsection{Error Term in the Off-diagonal Contribution}

In the following lemmas, we will successively separate some terms from the off-diagonal contribution at the cost of admissible error terms. 

In order to extend the sum over $c$ in  \eqref{eq:M 11 OD} to all of $c$ with $p^\nu\mid c$, we establish the following lemma.

\begin{lem}\label{lem:extend c}
    Let $\epsilon>0$, $\delta=0, 1$, and $B=0, 1, 2$. For $M,N\ll p^{\nu(1+\epsilon)}$ and $C>2\sqrt{aeMN}$, we have
    \begin{align*}
    &\sum_{\substack{p^{\nu-\delta} \mid c, \\ c> C}}
    \frac{\phi(c)}{c^2}
    \sum_{n \geq 1}  \lambda_g(n) \lambda_g(aenp^B)
G_{M,N}(aenp^B, np^B)
\\&\hspace{2em}\ll_{\epsilon,\eta}   p^{\delta-B}p^{-\nu+\nu\epsilon} (ae)^{-1+\epsilon} \frac{C^{2}}{\sqrt{MN}}\left(\frac{\sqrt{aeMN}}{C}\right)^{k+k'}.
    \end{align*}
\end{lem}


\begin{proof}
Using \eqref{eqn:G-def} and applying the change of variable $x\to c^2x$ we get
    \begin{align}\label{eqn:extend-all-c}
  & \sum_{\substack{p^{\nu-\delta} \mid c, \\ c> C}}
    \frac{\phi(c)}{c^2}
      \sum_{n \geq 1}  \lambda_g(n) \lambda_g(aenp^B)
G_{M,N}(aenp^B, np^B) \nonumber   \\&=    
        \sum_{np^B \asymp N}  \lambda_g(n) \lambda_g(aenp^B) 
   \sum_{\substack{c > C, \\ p^{\nu-\delta} \mid c}}
        \phi(c) \nonumber\\&\hspace{2em}\times\int_0^\infty 
         J_{k-1}\big(4\pi\sqrt{aenp^Bx}\big)
        J_g\big(4\pi\sqrt{aenp^Bx}\big)
       F_{M, N}(c^2x, np^B)
            \mathop{dx}.
    \end{align}
By Lemma \ref{lem:F bound}, we have
    \begin{align*}
    F_{M, N}(c^2x, np^B)
    \ll_{\eta} \frac{ 1 }{\sqrt{MN}}
    \Big(\frac{p^\nu}{c^2 x}\Big)^A
    \Big(\frac{p^\nu }{adnp^B}\Big)^{A'}.
    \end{align*}
Since $c^2x\asymp M$, $np^{B}\asymp N$, and $M,N\ll p^{\nu(1+\epsilon)}$, it is appropriate in this regime to take $A=A'=0$.
Using $x\leq  \frac{2M}{c^2}<\frac{2M}{C^2}$ and $C>2\sqrt{aeMN}$, we see that $x< (2aeN)^{-1}<(aenp^B)^{-1}$. By Lemma \ref{lem:Bessel}, it follows that
\[
     J_{k-1}\big(4\pi\sqrt{aenp^Bx}\big)
     \ll
    \frac{(4\pi\sqrt{aenp^Bx})^{k-1}}{(1+4\pi\sqrt{aenp^Bx})^{k-1/2}}
    \ll (aenp^Bx)^{\frac{k-1}{2}}.
\]
Similarly, $J_g\big(4\pi\sqrt{aenp^Bx}\big)\ll (aenp^Bx)^{\frac{k'-1}{2}}$. Using these bounds in \eqref{eqn:extend-all-c}, we obtain
    \begin{align*}
   & \sum_{\substack{p^{\nu-\delta} \mid c, \\ c> C}}
    \frac{\phi(c)}{c^2}
    \sum_{n \geq 1}  \lambda_g(n) \lambda_g(aenp^B)
G_{M,N}(aenp^B, np^B)  
   \\& \ll_{\epsilon,\eta} \, p^{\delta-B}p^{-\nu+\nu\epsilon} (ae)^{-1+\epsilon}\frac{C^{2}}{\sqrt{MN}}\left(\frac{\sqrt{aeMN}}{C}\right)^{k+k'}.
 \end{align*}
\end{proof}

Now, by using the following lemma, we will be able to extend the double sum over $M, N $  in \eqref{eq:M 11 OD} back to all $M, N$ at the cost of an admissible error term.


\begin{lem}\label{lem:extend MN}
Let $\delta=0, 1$, and suppose that $\ell < p^\nu$. For any positive real number $\mathcal C$, we have
    \begin{align*}
    &\sum_{p^{\nu-\delta} \mid c} \frac{\phi(c)}{c^2} 
    \sum_{\max\{M, N\} \gg p^{\nu(1+\epsilon)}}
    \sum_{n\geq 1}
    \lambda_g(n) \lambda_g(aenp^B) G_{M,N}(aenp^B, np^B)
    \ll_{ \mathcal C,\eta}   p^{-\nu \mathcal C} .
    \end{align*}
\end{lem}


\begin{proof}
The proof is similar to that of Lemma \ref{restrict M N}. We will only prove the claim for the case $M \gg p^{\nu(1+\epsilon)}$ and $N \ll p^{\nu(1+\epsilon)}$. The other cases have a contribution of at most the same size. 

First, we consider the sum over $c\leq \sqrt{aeNM}$. Recall that
    \begin{align*}
    G_{M,N}(aenp^B, np^B)
    = \int_0^\infty F_{M, N}(x, np^B)
    J_{k-1}\Big(\frac{4\pi \sqrt{aenp^B x}}{c}\Big)
    J_{g}\Big(\frac{4\pi \sqrt{aenp^B x}}{c}\Big)
    \mathop{dx}.
    \end{align*}
By Lemmas \ref{lem:Bessel} and \ref{lem:F bound}, we have 
    \[
     J_{k-1}\Big(\frac{4\pi \sqrt{aenp^B x}}{c}\Big)J_{g}\Big(\frac{4\pi \sqrt{aenp^B x}}{c}\Big)
  \ll
     \frac{(\sqrt{aenp^B x}/c)^{k+k'-2}}
     {(1+\sqrt{aenp^Bx}/c)^{k+k'-1}} \ll \bigg(\frac{\sqrt{aenp^Bx}}{c}\bigg)^{-1}
    \]
and    
    \[
    F_{M, N}(x, np^B)
    \ll_{\eta}  \frac{1}{\sqrt{MN}}
    \Big(\frac{p^\nu }{x}\Big)^{A}
    \Big(\frac{p^\nu}{adnp^B}\Big)^{A'} \,\, \text{for} \,\, x\in \Big[\frac{M}{2}, 2M\Big], np^B\in \Big[\frac{N}{2}, 2N\Big].
    \]
 It follows that       
 \[
    G_{M,N}(aenp^B, np^B)    \ll_{\eta}  c p^{\nu(A+A')}(ad)^{-A'}(ae)^{-\frac12}M^{-A}N^{-A'-1} , 
    \] 
    and so
    \begin{align*}
 & \sum_{\substack{M \gg p^{\nu(1+\epsilon)}\\N\ll p^{\nu(1+\epsilon)}}}   \sum_{\substack{p^{\nu-\delta} \mid c\\c\leq\sqrt{aeMN}}} \frac{\phi(c)}{c^2} 
    \sum_{n\geq 1}
    \lambda_g(n) \lambda_g(aenp^B) G_{M,N}(aenp^B, np^B)\\&\ll_{\epsilon,\eta}   
   p^{\nu(A+A')-B}(ad)^{-A'}(ae)^{-\frac12+\epsilon}  \sum_{\substack{M \gg p^{\nu(1+\epsilon)}\\N\ll p^{\nu(1+\epsilon)}}}   M^{-A}N^{-A'+\epsilon}\sum_{\substack{p^{\nu-\delta} \mid c\\c\leq\sqrt{aeMN}}} \frac{\phi(c)}{c} \\&\ll_{\epsilon,\eta}  p^{\delta-B} p^{\nu(A+A'-1)}(ad)^{-A'}(ae)^{\epsilon}  \sum_{\substack{M \gg p^{\nu(1+\epsilon)}\\N\ll p^{\nu(1+\epsilon)}}}   M^{\frac12-A}N^{\frac12-A'+\epsilon}.   
   \end{align*}
By choosing $A'=0$ and $A> \frac{\mathcal{C}+1}{\epsilon}+2$ for a given $\mathcal{C}>0$, we obtain
   \begin{align*}
  \sum_{\substack{M \gg p^{\nu(1+\epsilon)}\\N\ll p^{\nu(1+\epsilon)}}}   \sum_{\substack{p^{\nu-\delta} \mid c\\c\leq\sqrt{aeMN}}} \frac{\phi(c)}{c^2} 
    \sum_{n\geq 1}
    \lambda_g(n) \lambda_g(aenp^B) G_{M,N}(aenp^B, np^B)   \ll_{\mathcal{C},\eta}  p^{-\nu \mathcal{C}}.
     \end{align*}       
Following the same argument for $c > \sqrt{aeNM}$ while using the bound  
\[
J_{k-1}\Big(\frac{4\pi \sqrt{aenp^B x}}{c}\Big)J_{g}\Big(\frac{4\pi \sqrt{aenp^B x}}{c}\Big)
  \ll\left(\sqrt{aenp^B x}/c\right)^{k+k'-2}
  \]
  yields  
  \begin{align*}
  \sum_{\substack{M \gg p^{\nu(1+\epsilon)}\\N\ll p^{\nu(1+\epsilon)}}}   \sum_{\substack{p^{\nu-\delta} \mid c\\c>\sqrt{aeMN}}} \frac{\phi(c)}{c^2} 
    \sum_{n\geq 1}
    \lambda_g(n) \lambda_g(aenp^B) G_{M,N}(aenp^B, np^B)   \ll_{\mathcal{C},\eta} p^{-\nu \mathcal{C}},
     \end{align*}   
 for any $\mathcal{C}>0$.
  \end{proof}

By using the following lemma, we will be able to further simplify \eqref{eq:M 11 OD} by showing that the sum $\sum_{M, N}  \eta_M(x) \eta_N(adnp^B) $ can be approximated by $1$ on average.


\begin{lem}\label{lem:remove-dyadic}
Let $\delta= 0, 1$. Provided that $\ell< p^{\nu}$, we have for each $B=0, 1, 2$, 
    \begin{align*}
     \sum_{p^{\nu-\delta}  \mid c}
    &\frac{\phi(c)}{c^2}
    \sum_{n\geq 1}
    \lambda_g(n)\lambda_g(aenp^B) 
     \int_0^\infty 
    \frac{1}{\sqrt{xnp^B}} 
    V_{g, \omega_1}\Big(\tfrac{x}{p^\nu }\Big) 
    V_{g, \omega_2}\Big(\tfrac{adnp^B}{p^\nu }\Big)
    \\
    \times &
    \Big(1-\sum_{M, N}  \eta_M(x) \eta_N(np^B) \Big)
    J_{k-1}\Big(\tfrac{4\pi\sqrt{aenp^Bx}}{c}\Big)
    J_g\Big(\tfrac{4\pi\sqrt{aenp^Bx}}{c} \Big) 
    \mathop{dx} 
    \ll_{\epsilon, \eta}  (ae)^{\epsilon}p^{-\frac{\nu}{2}+\nu\epsilon}  .
    \end{align*}
\end{lem}


\begin{proof}
Observe that the integral in the given expression is supported over $[0, 1]$, the only interval where the integrand is possibly nonzero. We split the sums over $c$ and $n$ into suitable ranges so that Lemma \ref{V bounds} and Lemma \ref{lem:Bessel} may be applied uniformly. The claimed bound then follows by applying these estimates on each range.   
\end{proof}


\subsection{Main Term in the Off-diagonal Contribution}\label{sec:main term OD}

By Lemmas \ref{lem:extend c}, \ref{lem:extend MN}, and \ref{lem:remove-dyadic}, we can now write   \begin{equation}\label{eq:M1OD setup}
    \begin{split}
    &M^{OD}_1(B) = M^{OD}_{1, 1}(B)+M^{OD}_{1, 2}(B)
    \\
    = \, & \frac{2\pi i^{-k}}{p^{\nu(\omega_1+\omega_2)}}
    \sum_{de=\ell}
    \frac{1}{\sqrt d}
    \sum_{ab=d}  \frac{\mu(a) \lambda_g(b)}{\sqrt{a}}
    \sum_n  \lambda_g(n)  \lambda_g(aenp^B)
    \Big\{\sum_{p^\nu \mid c} \frac{\phi(c)}{c^2}
    -\frac{1}{p} \sum_{p^{\nu-1} \mid c} \frac{\phi(c)}{c^2}\Big\}
    \\
    &\hspace{2em} \times \int_0^\infty \frac{1}{\sqrt{xnp^B}}
    V_{g, \omega_1}\Big(\frac{x}{p^\nu }\Big)
    V_{g, \omega_2}\Big(\frac{adnp^B}{p^\nu }\Big)
        J_{k-1}\Big(\frac{4\pi \sqrt{aenp^B x}}{c}\Big)
        J_{g}\Big(\frac{4\pi \sqrt{aenp^B x}}{c}\Big)
        \mathop{dx} \\
        &\hspace{3em}+ O\bigg(p^{1-\nu(1-\epsilon)}
    \sum_{\substack{de=\ell\\ab=d}}  \frac{(ae)^{-1+\epsilon}b^{\epsilon}}{\sqrt{ad}}\sum_{M,N\ll p^{\nu(1+\epsilon)}}\frac{C^2}{\sqrt{MN}}\Big(\frac{\sqrt{aeMN}}{C}\Big)^{k+k'}\bigg)\\&\hspace{3em}+O\left(\ell^{\epsilon}p^{-\nu\left(\frac{1}{2}+\epsilon\right)}\right).
    \end{split}
    \end{equation}    
Using the change of variable
$x=\frac{c^2y^2}{(4\pi)^2 a e np^B}$, we get
    \begin{align*}
  & \int_0^\infty \frac{1}{\sqrt{xnp^B}}
   V_{g, \omega_1}\Big(\frac{x}{p^\nu}\Big)
   V_{g, \omega_2}\Big(\frac{adnp^B}{p^\nu}\Big)
    J_{k-1}\Big(\frac{4\pi \sqrt{aenp^Bx}}{c}\Big)
    J_{g}\Big(\frac{4\pi \sqrt{aenp^B x}}{c}\Big) \mathop{dx}
\\&=    \int_0^\infty 
  \frac{c} {2\pi \sqrt{ae} np^B}  
   V_{g, \omega_1}\Big(\frac{c^2y^2}{(4\pi)^2 ae np^B p^\nu }\Big) 
  V_{g, \omega_2}\Big(\frac{adnp^B}{p^\nu }\Big)
    J_{k-1}(y) 
    J_g( y) \mathop{dy} .
    \end{align*} 
Introducing the notation 
    \begin{align}\label{eqn:fancyZ}
    \mathcal{Z}(u)
    &:= 2\pi i^{-k}
    \int_0^\infty J_{k-1}(y) J_g(y)\nonumber \\
    &\times \bigg\{\sum_{p^\nu \mid c} 
    \frac{\phi(c)}{c} V_{g, \omega_1}\left(\frac{c^2y^2}{(4\pi)^2 ae u p^\nu  }\right)
    - \frac{1}{p}\sum_{p^{\nu-1} \mid c}
    \frac{\phi(c)}{c} V_{g, \omega_1}\left(\frac{c^2y^2}{(4\pi)^2 ae u p^\nu }\right)
    \bigg\}
    \mathop{dy}
    \end{align}
in \eqref{eq:M1OD setup}, we can write
    \begin{align}\label{eq:M1OD} 
    M^{OD}_1(B)
   & =   
    \frac{p^{-\nu(\omega_1+\omega_2)}}{2\pi}
    \sum_{de=\ell}
    \frac{1}{\sqrt{\ell}}
    \sum_{ab=d} \frac{\mu(a)  \lambda_g(b)}{a}
    \sum_n  \frac{\lambda_g(n)  \lambda_g(aenp^B)}{np^B  }
    V_{g, \omega_2} \Big(\frac{adnp^B}{p^\nu }\Big)
    \mathcal{Z}(np^B)\nonumber 
    \\&\hspace{2em}+ O\bigg(p^{1-\nu(1-\epsilon)}
    \sum_{\substack{de=\ell\\ab=d}}  \frac{(ae)^{-1+\epsilon}b^{\epsilon}}{\sqrt{ad}}\sum_{M,N\ll p^{\nu(1+\epsilon)}}\frac{C^2}{\sqrt{MN}}\Big(\frac{\sqrt{aeMN}}{C}\Big)^{k+k'}\bigg)\nonumber
    \\&\hspace{2em}+O\Big(\ell^{\epsilon}p^{-\nu\left(\frac{1}{2}+\epsilon\right)}\Big).
    \end{align}
Going back to the expression $\mathcal{Z}(u)$, we substitute \eqref{eqn:V-g-omega} into \eqref{eqn:fancyZ} and observe that 
 \begin{align*}
    \sum_{p^\nu \mid c} 
    \frac{\phi(c)}{c^{1+2s}}
    -\frac{1}{p} \sum_{p^{\nu-1} \mid c}
      \frac{\phi(c)}{c^{1+2s}}
      =\frac{1-\frac1{p}}{p^{2\nu s}} \frac{1-p^{2s-1}}{1-p^{-2s-1}}  \frac{\zeta(2s)}{ \zeta(1+2s) } .
    \end{align*}
It follows that
    \begin{align*}
    &\mathcal{Z}(u)= 
     \frac{2\pi i^{-k} (4\pi^2)^{\omega_1}(1-\frac1{p})}{\xi(\frac12)^5 P_g(\omega_1) \Gamma_g\big(\frac12+\omega_1\big)}
    \int_0^\infty J_{k-1}(y) J_g(y) 
     \\
    &\times \frac{1}{2\pi i}  \int_{(3)} 
    \Gamma_g\Big(\frac12+s\Big) (1-p^{2s-1}) \zeta(2s)
 \xi\Big(\frac12+s-\omega_1\Big)^5 \left(\frac{4aeu}{y^2p^\nu}\right)^s
    P_g(s)
    \frac{\mathop{ds}}{s-\omega_1}
    \mathop{dy}.
    \end{align*}
By \cite[p.~403]{Watson}, we know that
    \begin{align*}
    &\int_0^\infty 
    J_{k-1}(y) J_g(y) y^{-2s} \mathop{dy}
    = i^k \sqrt{\pi} 
    \frac{\Gamma_g(\frac12-s) \Gamma(s)}{\Gamma_g(\frac12+s) \Gamma(\frac12-s)} 
    \end{align*}
 holds for $\Re(k+k'-1)>\Re(2s) >0$. Hence,
    \begin{align*}
    \mathcal{Z}(u)
    &= 
    \frac{2\pi \sqrt{\pi} (4\pi^2)^{\omega_1} \left(1-\frac{1}{p}\right)}{P_g(\omega_1) \Gamma_g(\frac12+\omega_1) \xi(\frac12)^5} 
     \frac{1}{2\pi i} \int_{(3)} 
     \frac{\Gamma_g(\frac12-s) \Gamma(s)}{\Gamma(\frac12-s) }
     (1-p^{2s-1}) \zeta(2s) 
    \\
    &\hspace{3em} \times  \xi\Big(\frac12+s-\omega_1\Big)^5
    \left(\frac{4aeu}{p^\nu}\right)^{s}  P_g(s)  
    \frac{\mathop{ds}}{s-\omega_1} .
    \end{align*}
Then, by using the functional equation of $\zeta(s)$ in the form $ \displaystyle 
\Gamma(s)\zeta(2s)= \pi^{-\frac12+2s} \Gamma\big(1/2-s\big) \zeta(1-2s)$, we can rewrite $\mathcal{Z}(u)$ as 
    \begin{align}\label{eqn:Z(u)}
    \mathcal{Z}(u)
    &=
    \frac{2\pi(4\pi^2)^{\omega_1}  \big(1-\frac{1}{p}\big)}{P_g(\omega_1) \Gamma_g(\frac12+\omega_1) \xi(\frac12)^5} 
     \frac{1}{2\pi i} \int_{(3)} 
    \Gamma_g\Big(\frac12-s\Big)
     (1-p^{2s-1}) \zeta(1-2s)   \nonumber    \\
    &\hspace{3em} \times 
    \xi\Big(\frac12+s-\omega_1\Big)^5
   \left(\frac{4\pi^2aeu}{p^\nu}\right)^{s} 
  P_g(s)  
   \frac{\mathop{ds}}{s-\omega_1} .
    \end{align}
We now move the contour of integration in \eqref{eqn:Z(u)} to $\Re(s) =-3$ encountering residues of the integrand at $s=0$ and $s=\omega_1$. This yields
    \begin{equation}\label{eq:Zu-1}
    \begin{split}
    \mathcal{Z}(u)
    = &\, \frac{\pi}{\omega_1}    \Big(1-\frac{1}{p}\Big)^2
    \frac{(4\pi^2)^{\omega_1}P_g(0) \Gamma_g\big(\frac12\big)  \xi\big(\frac12-\omega_1\big)^5}{P_g(\omega_1) \Gamma_g(\frac12+\omega_1) \xi(\frac12)^5}  
     \\
    & + 2\pi   \Big(1-\frac{1}{p}\Big)
    \frac{(4\pi^2)^{\omega_1}\Gamma_g(\frac12-\omega_1)}{\Gamma_g(\frac12+\omega_1) } 
     \Big(\frac{4\pi^2 ae  u}{p^\nu}\Big)^{\omega_1} 
    \zeta(1-2\omega_1)  (1-p^{-1+2\omega_1})
    \\
    &+2\pi   \Big(1-\frac{1}{p}\Big)
    \frac{(4\pi^2)^{\omega_1}}{P_g(\omega_1) \Gamma_g(\frac12+\omega_1) \xi(\frac12)^5} 
     \frac{1}{2\pi i}  \int_{(-3)} 
    \Gamma_g\Big(\frac12-s\Big)
     (1-p^{2s-1}) \zeta(1-2s)       \\
    &\hspace{3em} \times 
    \xi\Big(\frac12+s-\omega_1\Big)^5
   \left(\frac{4\pi^2aeu}{p^\nu}\right)^{s} 
  P_g(s)  
   \frac{\mathop{ds}}{s-\omega_1}  .
    \end{split}
    \end{equation}
 From the change of variables $s\to -s$, the evenness of $P_g$ and the symmetry of $\xi$, it easily follows that
    \begin{align*}
    &2\pi   \Big(1-\frac{1}{p}\Big)
    \frac{(4\pi^2)^{\omega_1}}{P_g(\omega_1) \Gamma_g(\frac12+\omega_1) \xi(\frac12)^5} 
     \frac{1}{2\pi i}  \int_{(-3)} 
    \Gamma_g\Big(\frac12-s\Big)
     (1-p^{2s-1}) \zeta(1-2s)       \\
    &\hspace{3em} \times 
    \xi\Big(\frac12+s-\omega_1\Big)^5
   \left(\frac{4\pi^2aeu}{p^\nu}\right)^{s} 
  P_g(s)  
   \frac{\mathop{ds}}{s-\omega_1}  .   \\&   
    =-2\pi    \Big(1-\frac1{p}\Big)
    \frac{(4\pi^2)^{2\omega_1} \Gamma_g(\frac12-\omega_1)}{ \Gamma_g(\frac12+\omega_1) } 
    V_{g, -\omega_1}\Big(\frac{aeu}{p^\nu}\Big) ,
    \end{align*} where the last equality follows from applying \eqref{eqn:V-g-omega}. 
Therefore, by \eqref{eq:Zu-1} we have shown that
    \begin{equation}\label{eq:Zu-final}
    \begin{split}
    \mathcal{Z}(u)
    = &\, \frac{\pi}{\omega_1}    \Big(1-\frac{1}{p}\Big)^2
    \frac{(4\pi^2)^{\omega_1}P_g(0) \Gamma_g\big(\frac12\big)  \xi\big(\frac12-\omega_1\big)^5}{P_g(\omega_1) \Gamma_g(\frac12+\omega_1) \xi(\frac12)^5}  
     \\
    & + 2\pi   \Big(1-\frac{1}{p}\Big)
    \frac{(4\pi^2)^{2\omega_1}\Gamma_g(\frac12-\omega_1)}{\Gamma_g(\frac12+\omega_1) } 
     \Big(\frac{ ae  u}{p^\nu}\Big)^{\omega_1} 
    \zeta(1-2\omega_1)  (1-p^{-1+2\omega_1})
    \\
    &-2\pi    \Big(1-\frac1{p}\Big)
    \frac{(4\pi^2)^{2\omega_1} \Gamma_g(\frac12-\omega_1)}{ \Gamma_g(\frac12+\omega_1) } 
    V_{g, -\omega_1}\Big(\frac{aeu}{p^\nu}\Big) .
    \end{split}
    \end{equation}
 Applying \eqref{eq:Zu-final} in \eqref{eq:M1OD}  and plugging the resulting expression into to the identity \[M^{OD}_1=M^{OD}_1(0)-\lambda_f(p)M_{1}^{OD}(1)+M_{1}^{OD}(2)\] gives
    \begin{align}\label{eq:M1OD residues} 
    M^{OD}_1
    &=  
    \frac{p^{-\nu(\omega_1+\omega_2)}}{2\pi}
    \sum_{de=\ell}
    \frac{1}{\sqrt{\ell}}
    \sum_{ab=d} \frac{\mu(a)  \lambda_g(b)}{a}
    \sum_{n: (n,p)=1}  \frac{\lambda_g(n)  \lambda_g(aen)}{n }
    V_{g, \omega_2} \Big(\frac{adn}{p^\nu }\Big)
  \nonumber  \\
    &\hspace{2em}\times \bigg\{
    \frac{\pi}{\omega_1}    \Big(1-\frac{1}{p}\Big)^2
    \frac{(4\pi^2)^{\omega_1}P_g(0) \Gamma_g\big(\frac12\big)  \xi\big(\frac12-\omega_1\big)^5}{P_g(\omega_1) \Gamma_g(\frac12+\omega_1) \xi(\frac12)^5}  
\nonumber    \\
    & \hspace{3em} + 2\pi   \Big(1-\frac{1}{p}\Big)
    \frac{(4\pi^2)^{2\omega_1}\Gamma_g(\frac12-\omega_1)}{\Gamma_g(\frac12+\omega_1) } 
    \Big(\frac{ ae n}{p^\nu}\Big)^{\omega_1} 
    \zeta(1-2\omega_1)  (1-p^{-1+2\omega_1})
  \nonumber  \\
    &\hspace{3em} -2\pi    \Big(1-\frac1{p}\Big)
    \frac{(4\pi^2)^{2\omega_1} \Gamma_g(\frac12-\omega_1)}{ \Gamma_g(\frac12+\omega_1) } 
    V_{g, -\omega_1}\Big(\frac{aen}{p^\nu}\Big) 
    \bigg\}
    \\&\hspace{1em} + O\bigg(p^{1-\nu(1-\epsilon)}
    \sum_{\substack{de=\ell\\ab=d}}  \frac{(ae)^{-1+\epsilon}b^{\epsilon}}{\sqrt{ad}}\sum_{M,N\ll p^{\nu(1+\epsilon)}}\frac{C^2}{\sqrt{MN}}\Big(\frac{\sqrt{aeMN}}{C}\Big)^{k+k'}\bigg)\nonumber
    \\&\hspace{1em}+O\Big(\ell^{\epsilon}p^{-\nu\left(\frac{1}{2}+\epsilon\right)}\Big) . \nonumber
   \end{align}
For convenience, we set
   \begin{align*}
   M^{OD}_1 &= \mathscr{M}^{OD}_{1;1} +\mathscr{M}^{OD}_{1;2}+\mathscr{M}^{OD}_{1;3}\\&\hspace{2em} 
   + O\bigg(p^{1-\nu(1-\epsilon)}
    \sum_{\substack{de=\ell\\ab=d}}  \frac{(ae)^{-1+\epsilon}b^{\epsilon}}{\sqrt{ad}}\sum_{M,N\ll p^{\nu(1+\epsilon)}}\frac{C^2}{\sqrt{MN}}\Big(\frac{\sqrt{aeMN}}{C}\Big)^{k+k'}\bigg)
    \\&\hspace{2em}+O\Big(\ell^{\epsilon}p^{-\nu\left(\frac{1}{2}+\epsilon\right)}\Big). 
 \end{align*}
 Here, 
 \begin{align*}
 \mathscr{M}^{OD}_{1;1}&=  \frac{p^{-\nu(\omega_1+\omega_2)}}{2\pi} \frac{\pi}{\omega_1}    \Big(1-\frac{1}{p}\Big)^2
    \frac{(4\pi^2)^{\omega_1}P_g(0) \Gamma_g\big(\frac12\big)  \xi\big(\frac12-\omega_1\big)^5}{P_g(\omega_1) \Gamma_g(\frac12+\omega_1) \xi(\frac12)^5}   \nonumber  \\
    &\hspace{2em}\times
\sum_{de=\ell}
    \frac{1}{\sqrt{\ell}}
    \sum_{ab=d} \frac{\mu(a)  \lambda_g(b)}{a}
    \sum_{n: (n,p)=1}  \frac{\lambda_g(n)  \lambda_g(aen)}{n }
    V_{g, \omega_2} \Big(\frac{adn}{p^\nu }\Big),
    \\
     \mathscr{M}^{OD}_{1;2}
     &= p^{-\nu(\omega_1+\omega_2)}p^{-\nu\omega_1}  \Big(1-\frac{1}{p}\Big)
    \frac{(4\pi^2)^{2\omega_1}\Gamma_g(\frac12-\omega_1)}{\Gamma_g(\frac12+\omega_1) } 
    \zeta_{p^{\nu}}(1-2\omega_1)  \nonumber\\&\hspace{2em}\times
    \sum_{de=\ell}
    \frac{1}{\sqrt{\ell}}
    \sum_{ab=d} \frac{\mu(a)  \lambda_g(b)(ae)^{\omega_1}}{a}
    \sum_{n: (n,p)=1}  \frac{\lambda_g(n)  \lambda_g(aen)}{n^{1-\omega_1} }
    V_{g, \omega_2} \Big(\frac{adn}{p^\nu }\Big),
 \end{align*}
 and 
 \begin{align*}
  \mathscr{M}^{OD}_{1;3}
  &=p^{-\nu(\omega_1+\omega_2)}   \Big(1-\frac1{p}\Big)
    \frac{(4\pi^2)^{2\omega_1} \Gamma_g(\frac12-\omega_1)}{ \Gamma_g(\frac12+\omega_1) } 
 \\
    &\hspace{2em}\times  \sum_{de=\ell}
    \frac{1}{\sqrt{\ell}}
    \sum_{ab=d} \frac{\mu(a)  \lambda_g(b)}{a}
    \sum_{n: (n,p)=1}  \frac{\lambda_g(n)  \lambda_g(aen)}{n }
    V_{g, \omega_2} \Big(\frac{adn}{p^\nu }\Big)   V_{g, -\omega_1}\Big(\frac{aen}{p^\nu}\Big).
 \end{align*}

A similar decomposition can be obtained for $M_{2}^{OD}$,  $M_{3}^{OD}$ and  $M_{4}^{OD}$. 
In analyzing $\mathscr{M}_{1;1}^{OD}$, we open up $V_{g, \omega_2} \big(\frac{adn}{p^\nu }\big)$ using its definition in \eqref{eqn:V-g-omega}. We move the contour of integration from $\Re(s)=3$ to $\Re(s)=-\frac12$ and collect the residues of the integrand at $s=0$ and $s= \omega_2$ which we denote by $\mathscr{M}_{1;1}^{OD,(0)}$ and $\mathscr{M}_{1;1}^{OD,(\omega_2)}$ respectively.  This  allows us to write $\mathscr{M}_{1;1}^{OD}$ as 
\[\mathscr{M}_{1;1}^{OD,(0)}+\mathscr{M}_{1;1}^{OD,(\omega_2)}+\mathscr{E}_{1;1}^{OD}.\] Here, the term $\mathscr{E}_{1;1}^{OD}$ is the resulting contour integral along $\Re(s)=-\frac12$. The corresponding results for  $\mathscr{M}_{2;1}^{OD}, \mathscr{M}_{3;1}^{OD}$ and $\mathscr{M}_{4;1}^{OD}$, which appear in $M_{2}^{OD}$,  $M_{3}^{OD}$ and  $M_{4}^{OD}$, respectively, can be obtained by observing the symmetry between the expressions $\mathscr{M}_{j;1}^{OD}$. For example, in light of \eqref{eq:Mg setup}, we notice that $\mathscr{M}^{OD}_{2;1}$ can be obtained by applying the swap $\omega_1 \to -\omega_1$ to all terms in $\mathscr{M}_{1;1}^{OD}$ except $p^{-\nu(\omega_1+\omega_2)}$. A straightforward computation shows that 
\[
 \mathscr{M}_{1;1}^{OD,(0)}
     +\ve_{\omega_1}(f\otimes g)\mathscr{M}_{2;1}^{OD,(0)}
  +\ve_{\omega_2}(f\otimes g)\mathscr{M}_{3;1}^{OD,(0)}
      +\ve_{\omega_1}(f\otimes g)\ve_{\omega_2}(f\otimes g) \mathscr{M}_{4;1}^{OD,(0)}=0.
\]
Hence, the contribution of the residues at $s=0$ to the off-diagonal expression $M^{OD}$ given in \eqref{eqn:off-diag-contr} is zero. It follows that the contribution to $M^{OD}$ coming from the $\mathscr{M}_{j;1}^{OD}$ for $j=1,\cdots,4$ is
\begin{equation}\label{eqn:fancyMOD11}
\begin{split}
&\mathscr{M}_{1;1}^{OD}
     +\ve_{\omega_1}(f\otimes g)\mathscr{M}_{2;1}^{OD}
  +\ve_{\omega_2}(f\otimes g)\mathscr{M}_{3;1}^{OD}
      +\ve_{\omega_1}(f\otimes g)\ve_{\omega_2}(f\otimes g) \mathscr{M}_{4;1}^{OD} \\
      &= \mathscr{M}_{1;1}^{OD,(\omega_2)}
     +\ve_{\omega_1}(f\otimes g)\mathscr{M}_{2;1}^{OD,(\omega_2)}
  +\ve_{\omega_2}(f\otimes g)\mathscr{M}_{3;1}^{OD,(-\omega_2)}
  \\
     &\qquad +\ve_{\omega_1}(f\otimes g)\ve_{\omega_2}(f\otimes g) \mathscr{M}_{4;1}^{OD,(-\omega_2)} 
    +O\Big(\max_{j}\left|\mathscr{E}_{j;1}^{OD}\right|\Big).
\end{split}
\end{equation}
Observe that $\mathscr{E}_{j;1}^{OD}=O(\ell^{\epsilon}p^{-\frac{\nu}{2}})$ for each $j$. This can be seen by using the following lemma regarding the series \[\sum_{n, (n, p)=1}  \frac{\lambda_g(n)  \lambda_g(aen)}{n^{z}}\] which we encounter when interchanging summation and integration in $\mathscr{E}_{j;1}^{OD}$.

\begin{lem}
For $\Re(z)>1$, we have 
 \begin{equation*}
    \sum_{n, (n, p)=1}  \frac{\lambda_g(n)  \lambda_g(aen)}{n^{z}}
    =
    \Upsilon(z, p, ae) \frac{L(z, g\otimes g)}{\zeta(2z)},
    \end{equation*}
where \[
    \Upsilon(z, p, ae)
    =  \Big( \sum_{j=0}^\infty
    \frac{\lambda_g(p^j)^2}{p^{zj}}
    \Big)^{-1}
        \prod_{q\mid ae} 
         \Big( \sum_{j=0}^\infty
    \frac{\lambda_g(q^j)^2}{q^{zj}}
    \Big)^{-1}
     \Big( \sum_{j=0}^\infty
    \frac{\lambda_g(q^j) \lambda_g(q^{\nu_{q}(ae)+j})}{q^{zj}}
    \Big)
    \]
is holomorphic and satisfies  $\left|\Upsilon(z, p, ae)\right|\ll_{\epsilon} (ae)^\epsilon$ in every fixed domain  $\Re(z)\geq\epsilon>0$. 
\end{lem}
\begin{proof}
For $\Re(z)>1$, we write   \begin{align*}
    \sum_{n, (n, p)=1}  \frac{\lambda_g(n)  \lambda_g(aen)}{n^{z}}
    &= \sum_{n}  \frac{\lambda_g(n)  \lambda_g(aen)}{n^{z}}
    \Big( \sum_{j=0}^\infty
    \frac{\lambda_g(p^j) \lambda_g(p^{\nu_p(ae)+j})}{p^{zj}}
    \Big)^{-1} \\&=  \sum_{n}  \frac{\lambda_g(n)  \lambda_g(aen)}{n^{z}}
    \Big( \sum_{j=0}^\infty
    \frac{\lambda_g(p^j)^2 }{p^{zj}}
    \Big)^{-1}   \end{align*}
since  $\nu_p(ae)=0$ provided that $(\ell, p)=1$. We further write    \[
     \sum_{n}  \frac{\lambda_g(n)  \lambda_g(aen)}{n^{z}}
     =\prod_{(q, ae)=1} 
     \Big( \sum_{j=0}^\infty
    \frac{\lambda_g(q^j)^2}{q^{zj}}
    \Big)
    \prod_{q\mid ae} 
     \Big( \sum_{j=0}^\infty
    \frac{\lambda_g(q^j) \lambda_g(q^{\nu_{q}(ae)+j})}{q^{zj}}
    \Big).
    \]
Now the definition of the $L$-function combined with this identity give
    \begin{align*}
    L(z, g\otimes g)
   & = \zeta(2z) \sum_{n=1}^\infty \frac{\lambda_g(n)^2}{n^z}
    \\
    &= \zeta(2z) 
    \prod_{q \mid ae} 
     \Big( \sum_{j=0}^\infty
    \frac{\lambda_g(q^j)^2}{q^{zj}}
    \Big)
    \prod_{(q, ae)=1} 
     \Big( \sum_{j=0}^\infty
    \frac{\lambda_g(q^j)^2}{q^{zj}}
    \Big)
    \\
    &= \zeta(2z) 
    \prod_{q \mid ae} 
     \Big( \sum_{j=0}^\infty
    \frac{\lambda_g(q^j)^2}{q^{zj}}
    \Big)
    \prod_{q\mid ae} 
     \Big( \sum_{j=0}^\infty
    \frac{\lambda_g(q^j) \lambda_g(q^{\nu_{q}(ae)+j})}{q^{zj}}
    \Big)^{-1}
     \sum_{n}  \frac{\lambda_g(n)  \lambda_g(aen)}{n^{z}} .
    \end{align*}
Thus, for $\Re(z) >1$, we have
    \begin{align*}
       & \sum_{n, (n, p)=1}  \frac{\lambda_g(n)  \lambda_g(aen)}{n^{z}}
       \\
        & = 
        \frac{L(z, g\otimes g)}{\zeta(2z) } \Big( \sum_{j=0}^\infty
    \frac{\lambda_g(p^j)^2}{p^{zj}}
    \Big)^{-1}
        \prod_{q\mid ae} 
         \Big( \sum_{j=0}^\infty
    \frac{\lambda_g(q^j)^2}{q^{zj}}
    \Big)^{-1}
     \Big( \sum_{j=0}^\infty
    \frac{\lambda_g(q^j) \lambda_g(q^{\nu_{q}(ae)+j})}{q^{zj}}
    \Big).
    \end{align*}
\end{proof}

Similarly, for each $\mathscr{M}^{OD}_{j;2}$, $1\leq j \leq 4$, we move the contour $\Re(s)=3$ to $\Re(s)=-\frac12$ and collect the residues at $s=0, \pm \omega_2$. The  contribution of the residues at $s=0$ to  $M^{OD}$ is also zero. It follows that the contribution to $M^{OD}$ coming from $\mathscr{M}_{j;2}^{OD}$ for $j=1,\cdots,4$ is
\begin{equation}\label{eqn:fancyMOD12}
\begin{split}
&\mathscr{M}_{1;2}^{OD}
     +\ve_{\omega_1}(f\otimes g)\mathscr{M}_{2;2}^{OD}
  +\ve_{\omega_2}(f\otimes g)\mathscr{M}_{3;2}^{OD}
      +\ve_{\omega_1}(f\otimes g)\ve_{\omega_2}(f\otimes g) \mathscr{M}_{4;2}^{OD}\\
      &= \mathscr{M}_{1;2}^{OD,(\omega_2)}
     +\ve_{\omega_1}(f\otimes g)\mathscr{M}_{2;2}^{OD,(\omega_2)}
  +\ve_{\omega_2}(f\otimes g)\mathscr{M}_{3;2}^{OD,(-\omega_2)}
  \\
      &\qquad+\ve_{\omega_1}(f\otimes g)\ve_{\omega_2}(f\otimes g) \mathscr{M}_{4;2}^{OD,(-\omega_2)}
      +O\big(\ell^{\epsilon}p^{-\frac{\nu}{2}}\big).
\end{split}
\end{equation}
Now, recall that we set 
   \begin{align}\label{eqn:MOD13}
    \mathscr{M}^{OD}_{1;3}
    &=   -  p^{-\nu(\omega_1+\omega_2)}  \Big(1-\frac1{p}\Big)
    \frac{(4\pi^2)^{2\omega_1} \Gamma_g(\frac12-\omega_1)}{ \Gamma_g(\frac12+\omega_1) }\nonumber
    \\
     &\quad \times   \sum_{de=\ell}
        \frac{1}{\sqrt{\ell}}
        \sum_{ab=d} \frac{\mu(a)  \lambda_g(b)}{a}
        \sum_{n, (n,p)=1}  \frac{\lambda_g(n)  \lambda_g(aen)}{n }
     V_{g, -\omega_1}\Big(\frac{aen}{p^\nu}\Big)  V_{g, \omega_2} \Big(\frac{adn}{p^\nu }\Big)  .
    \end{align}
Using symmetry and simplifying the resulting expressions we get
    \begin{align}\label{eqn:MOD23}
    \varepsilon_1(f\otimes g )
    \mathscr{M}^{OD}_{2; 3}
    =&- 
    p^{-\nu(\omega_1+\omega_2)}  \Big(1-\frac1{p}\Big)\nonumber\\&\hspace{2em}\times
     \sum_{de=\ell}
    \frac{1}{\sqrt{\ell}}
    \sum_{ab=d} \frac{\mu(a)  \lambda_g(b)}{a}
    \sum_{n, (n,p)=1}  \frac{\lambda_g(n)  \lambda_g(aen)}{n }
  V_{g, \omega_1}\Big(\frac{aen}{p^\nu}\Big)  V_{g, \omega_2} \Big(\frac{adn}{p^\nu }\Big)    ,
    \end{align}
    \begin{align}\label{eqn:MOD33}
    \varepsilon_2(f\otimes g )
    \mathscr{M}^{OD}_{3;3}
    =&- p^{-\nu(\omega_1+\omega_2)}  \Big(1-\frac1{p}\Big) (4\pi^2)^{2(\omega_1+\omega_2)} \frac{\Gamma_g(1/2-\omega_2)}{ \Gamma_g(1/2+\omega_2)}
    \frac{ \Gamma_g(\frac12-\omega_1)}{ \Gamma_g(\frac12+\omega_1) }\nonumber
    \\
    &\quad \times   \sum_{de=\ell}
    \frac{1}{\sqrt{\ell}}
    \sum_{ab=d} \frac{\mu(a)  \lambda_g(b)}{a}
    \sum_{n, (n,p)=1}  \frac{\lambda_g(n)  \lambda_g(aen)}{n }
   V_{g, -\omega_1}\Big(\frac{aen}{p^\nu}\Big)  V_{g, -\omega_2} \Big(\frac{adn}{p^\nu }\Big)     ,
    \end{align}
and    \begin{align}\label{eqn:MOD43}
      \varepsilon_1(f\otimes g ) \varepsilon_2(f\otimes g )
    \mathscr{M}^{OD}_{4;3}
    &=- p^{-\nu(\omega_1+\omega_2)}  \Big(1-\frac1{p}\Big)\frac{(4\pi^2)^{2\omega_2} \Gamma_g(1/2-\omega_2)}{ \Gamma_g(1/2+\omega_2)}
   \nonumber \\
     \times  & \sum_{de=\ell}
    \frac{1}{\sqrt{\ell}}
    \sum_{ab=d} \frac{\mu(a)  \lambda_g(b)}{a}
    \sum_{n, (n,p)=1}  \frac{\lambda_g(n)  \lambda_g(aen)}{n }
    V_{g, \omega_1}\Big(\frac{aen}{p^\nu}\Big)  V_{g, -\omega_2} \Big(\frac{adn}{p^\nu }\Big) .
    \end{align}
 
 Finally, we make the crucial observation that 
    \begin{align}\label{eqn:MDcancellation}
    M^D_1
    &+ \mathscr{M}^{OD}_{1;3}
    +\varepsilon_1(f\otimes g ) \big(M^D_2+ \mathscr{M}^{OD}_{2;3}\big)
    +\varepsilon_2(f\otimes g ) \big(M^D_3+ \mathscr{M}^{OD}_{3;3}\big)
    \\
    &+\varepsilon_1(f\otimes g ) \varepsilon_2(f\otimes g )\big(M^D_4+ \mathscr{M}^{OD}_{4;3}\big)
    =0\nonumber.
    \end{align}
which follows directly from comparing the expressions in  \eqref{eq:MD1},  \eqref{eq:MD2},  \eqref{eq:MD3},  \eqref{eq:MD4}, \eqref{eqn:MOD13}, \eqref{eqn:MOD23},  \eqref{eqn:MOD33}, and  \eqref{eqn:MOD43}.

We set
    \begin{align*}
     &M^D  + M^{OD}
     = \big(M^D_{1}  + M^{OD}_{1}\big)
       +  \epsilon_{\omega_1}(f\otimes g)  \big(M^D_{2}  + M^{OD}_{2}\big)
       \\
      &  +  \epsilon_{\omega_2}(f\otimes g)  \big(M^D_{3}  + M^{OD}_{3}\big)
         + \epsilon_{\omega_1}(f\otimes g)  \epsilon_{\omega_2}(f\otimes g) 
          \big(M^D_{4}  + M^{OD}_{4}\big).
    \end{align*}
In view of \eqref{eqn:MDcancellation}, we see that once we add the total diagonal contribution $M^D$ to the off-diagonal contribution $M^{OD}$, only the contributions of the terms $\mathscr{M}^{OD}_{j;1}, \mathscr{M}^{OD}_{j;2}$ for $1\leq j\leq 4$ survive. These are given in \eqref{eqn:fancyMOD11} and \eqref{eqn:fancyMOD12}. By simplifying the remaining eight  terms and handling cancellations, we arrive at the following proposition.


    \begin{prop}\label{prop:MD plus MOD}
We have
    \begin{align*}
    M^D  + M^{OD}
   &=\, p^{-\nu(\omega_1+\omega_2)} 
   \Big( \mathfrak{m}(\omega_1, \omega_2 ) 
   + \epsilon_{\omega_1}(f\otimes g) \mathfrak{m}(-\omega_1, \omega_2 ) 
   +\epsilon_{\omega_2}(f\otimes g) \mathfrak{m}(\omega_1, -\omega_2 ) 
   \\
   &\qquad \qquad \qquad+\epsilon_{\omega_1}(f\otimes g)\epsilon_{\omega_2}(f\otimes g)
    \mathfrak{m}(-\omega_1, -\omega_2 ) 
   \Big)\\
   &\hspace{2em}+O\bigg(p^{1-\nu(1-\epsilon)}
    \sum_{\substack{de=\ell\\ab=d}}  \frac{(ae)^{-1+\epsilon}b^{\epsilon}}{\sqrt{ad}}\sum_{M,N\ll p^{\nu(1+\epsilon)}}\frac{C^2}{\sqrt{MN}}\Big(\frac{\sqrt{aeMN}}{C}\Big)^{k+k'}\bigg)\nonumber
    \\&\hspace{2em}+O\left(\ell^{\epsilon}p^{-\frac{\nu}{2}}\right), 
    \end{align*}
where
       \begin{align*}
    &\mathfrak{m}(\omega_1, \omega_2 )
    \\
    =&  \left(1-\frac{1}{p}\right)p^{-\nu\omega_1+\nu\omega_2}\frac{(4\pi^2)^{2\omega_1} }{2\omega_1}  \frac{\Gamma_g(1/2-\omega_1) }{\Gamma_g(1/2+\omega_1)}\frac{L(1-\omega_1+\omega_2, g\otimes g)}{\zeta(2-2\omega_1+2\omega_2)}
    \zeta_{p^\nu}(1-2\omega_1)\zeta_{p^\nu}(1+2\omega_2)\\&\hspace{2em}\times \sum_{de=\ell} 
        \frac{ e^{\omega_1} d^{-\omega_2}}{\sqrt{\ell}}
      \sum_{ab=d} \frac{\mu(a)  \lambda_g(b)}{a^{1-\omega_1+\omega_2} } 
      \Upsilon(1-\omega_1+\omega_2, p, ae) .
    \end{align*}
\end{prop}


\subsection{Bounding the Off-off-diagonal Contribution} \label{sec:ood}

We now estimate $M^{OOD}$ given by \eqref{eqn:off-off-diag-contr}. This estimation is crucial for determining  the error term in Theorem \ref{thm:2nd moment}. Recall that for each $j=1,2,3,4$, we set
    \begin{equation}\label{eqn:MOODj1}
     M^{OOD}_{j, 1}
     =  M^{OOD}_{1, 1}(0)
     -\lambda_g(p)  M^{OOD}_{j, 1}(1)
     + M^{OOD}_{j, 1}(2)
    \end{equation}
and
    \begin{equation}\label{eqn:MOODj2}
     M^{OOD}_{j, 2}
     =  M^{OOD}_{j, 2}(0)
     -\lambda_g(p)  M^{OOD}_{j, 2}(1)
     + M^{OOD}_{j, 2}(2),
    \end{equation}
    so that $M^{OOD}_j=M^{OOD}_{j, 1}+M^{OOD}_{j, 2}$.
    
We will work only on bounding $M^{OOD}_{1, 1}(B) $ for     $B=0, 1, 2$,  All the other terms can be bounded similarly. We have by \eqref{eq:M 11 OOD} 
     \begin{equation}\label{eq:M11OOD}
     \begin{split}
    M^{OOD}_{1, 1}(B) 
    & = 
     2\pi i^{-k} p^{-\nu(\omega_1+\omega_2)}
    \sum_{de=\ell} \frac{1}{\sqrt d} \sum_{ab=d} 
   \frac{ \mu(a) \lambda_g(b)}{\sqrt a}
    \sum_{M, N \ll p^{\nu(1+\epsilon)}}    \sum_{\substack{c\leq C,\\ p^\nu \mid c}} \frac{1}{c^2}
    \sum_{h\neq 0} S(0, h; c) T_{h}(c, B) .
    \end{split}
   \end{equation}
Recall that 
    \begin{align*}
    T_{h}(c, B)
    =  \sum_{n\geq 1}  
    \lambda_g(n) \lambda_g(aenp^B - h)
    G_{M,N}\left(aenp^B-h, np^B\right) ,
    \end{align*}
where
    \[
    G_{M,N}(z, y) =
    \int_0^\infty F_{M, N}(x, y) 
    J_{k-1}\Big(\frac{4\pi\sqrt{aexy}}{c}\Big)
    J_g\bigg(\frac{4\pi\sqrt{zx}}{c} \bigg) 
    \mathop{dx}.
    \]
Note that the expression for $T_{h}(c, B)$ is in the form \eqref{eq:DFI smoothed sum}.  We will thus estimate it by applying Theorem \ref{DFI type result} similarly to the discussion in~\cite[Section 7.2.1]{KMV RS}. To this end, we need to check that $G$ satisfies  condition  \eqref{eq:DFI H condition}. We choose
    \[
    H(z, y)= G_{M,N}\Big(z, \frac{y}{ae}\Big), 
    \quad
    \text{so that}
    \quad
    T_h(c, B)
    = \sum_{m-aenp^B  = -h} \lambda_g(m) \lambda_g(n) H(m, aenp^B) .
    \]
We also set the parameters
    \[
    P=1+\frac{\sqrt{aep^B MN}}{c}, 
    \quad 
    Z= \frac{c^2P^2}{M} ,
    \quad
    Y= aep^BN ,
    \]
where $Z>Y$. Then by~\cite[p. 153]{KMV RS} the function
    \[
    \frac{cH(z, y)}{\sqrt{ae}MP^{-\frac32}}
    \]
satisfies \eqref{eq:DFI H condition}. Then by Theorem \ref{DFI type result}, we have
    \begin{equation}\label{eq:Th bound}
    \begin{split}
    T_h(c, B)
    &\ll_{g, \epsilon}  \frac{\sqrt{ae}MP^{-\frac32}}{c} P^{\frac54} (Z+Y)^{\frac14} (YZ)^{\frac14+\epsilon}
    \ll_{g, \epsilon}  \frac{\sqrt{ae}M}{c} P^{-\frac14} Z^{\frac12+\epsilon} Y^{\frac14+\epsilon}
    \\
    &\ll_{g, \epsilon} (ae)^{3/4} c^{2\epsilon} P^{\frac34+2\epsilon} M^{\frac12-\epsilon}
    N^{\frac14+\epsilon} .
    \end{split}
    \end{equation}
    Next, we note the following elementary bound     \[
    S(0, h; c) = \frac{\mu(c) \phi(c)}{\phi\big(\frac{c}{(h, c)}\big)}
    \ll (h, c) .
    \]
Then by \eqref{eq:Th bound}, we have
    \begin{align}\label{eq:bound h leq Z}
    \sum_{|h|\ll Z^{1+\epsilon}} S(0, h ; c) T_h(c, B)
    &\ll
      (ae)^{\frac34} P^{\frac34+2\epsilon} M^{\frac12-\epsilon}
    N^{\frac14+\epsilon} c^{2\epsilon} 
    \sum_{h\ll Z^{1+\epsilon}} (h, c) \nonumber\\&\ll  (ae)^{\frac34}c^{3\epsilon} 
 P^{\frac34+2\epsilon} M^{\frac12-\epsilon}
    N^{\frac14+\epsilon}Z^{1+\epsilon}  .   \end{align}

For $|h| \gg Z^{1+\epsilon}$, we can assume without loss of generality that $m \gg Z^{1+\epsilon}$. Since
    \[
    H(m, aenp^B)
    \ll \frac{\sqrt{ae}MP^{-3/2}}{c}
    \Big(1+\frac{m}{Z}\Big)^{-A},
    \]
    $S(0, h; c) \leq \phi(c)$, and $\lambda_g(n) \ll_{\epsilon} n^{\epsilon}$, we have 
    \begin{equation}\label{eq:bound h geq Z} 
    \begin{split}
    \sum_{|h| \gg Z^{1+\epsilon}}  S(0, h; c) T_{h}(c, B)
      &\ll \frac{\phi(c)}{c}\sqrt{ae}MP^{-\frac32}
    \sum_{n\in [N/2, 2N]} 
    \sum_{\substack{m \gg Z^{1+\epsilon}, \\ m \in[M/2, 2M]}} m^{\epsilon}  
    n^{\epsilon} 
    \Big(\frac{m}{Z}\Big)^{-A} 
    \\ 
   & \ll  \sqrt{ae}MP^{-\frac32} N^{1+\epsilon}Z^{-100},
    \end{split}
    \end{equation}
    where the last bound is achieved by choosing $A$ large enough.

Since the bound in \eqref{eq:bound h geq Z} is negligible with respect to that obtained in \eqref{eq:bound h leq Z}, we get the estimate
    \begin{align*}
  \frac{1}{c^2}  \sum_{h\neq 0} S(0,h;c) T_h(c, B)
    &
    \ll (ae)^{\frac34} P^{\frac{11}4+4\epsilon} M^{-\frac12-2\epsilon}
    N^{\frac14+\epsilon} c^{5\epsilon} \\&
\ll  (ae)^{\frac34}  M^{-\frac12-2\epsilon}
    N^{\frac14+\epsilon} c^{5\epsilon}
 +  p^{B(\frac{11}{4}+2\epsilon)}(ae)^{\frac{17}{8}+\epsilon}  M^{\frac{7}{8}}
    N^{\frac{13}{8}+3\epsilon} c^{-\frac{11}{4}+\epsilon}  .
\end{align*}
Summing the above expression over $c$ gives
\[
\sum_{\substack{p^\nu \mid c, \\ c\leq C}}  \frac{1}{c^2} 
 \sum_{h\neq 0} S(0,h;c) T_h(c, B) 
 \ll 
 \frac{C^{1+5\epsilon}}{p^\nu} (ae)^{\frac34}  M^{-\frac12-2\epsilon}
    N^{\frac14+\epsilon} 
    +  (ae)^{\frac{17}{8}+\epsilon}  M^{\frac{7}{8}}
    N^{\frac{13}{8}+3\epsilon} p^{\nu(-\frac{11}{4}+\epsilon)}.
\]  
Comparing the first term on the right-hand side of the above inequality with the bound in Lemma \ref{restrict C} we see that the optimal choice for $C$ is 
\begin{equation}\label{eqn:choiceC}
C\asymp p^{\frac{\nu}{k-2\theta+4\epsilon}} (ae)^{\frac{2k-5}{4k-8\theta+16\epsilon}}M^{\frac{1}{2}} N^{\frac{2k-3-4\theta-4\epsilon}{4k-8\theta+16\epsilon}} .
\end{equation}
Assuming $\ell\ll p^{\nu\left(\frac{1}{5-4\theta}-O(\epsilon)\right)}$ guarantees that $C>\sqrt{aeMN}$. Hence,
\[
M_{1,1}^{OOD}(B) \ll\sum_{M, N \ll p^{\nu(1+\epsilon)}}   
\sum_{\substack{p^\nu \mid c, \\ c\leq C}}  \frac{1}{c^2} 
 \sum_{h\neq 0} S(0,h;c) T_h(c, B) 
 \ll 
\ell^{17/8}p^{-\nu\left(\frac14-O(\epsilon)\right)}
+
\ell^{\frac{5k-5-6\theta}{4k-8\theta}}
p^{-\nu\left(\frac{k-1-2\theta}{4k-8\theta}-O(\epsilon)\right)}.\]
The same bound holds for $M^{OOD}_{j,1}(B)$ and $M^{OOD}_{j,2}(B)$ for $j=2,3,4$. Hence, 
\begin{equation}\label{eqn:bound-MOOD}
M^{OOD}\ll \ell^{17/8}p^{-\nu\left(\frac14-O(\epsilon)\right)}
+
\ell^{\frac{5k-5-6\theta}{4k-8\theta}}
p^{-\nu\left(\frac{k-1-2\theta}{4k-8\theta}-O(\epsilon)\right)}.
\end{equation}

\section{Proof of Theorem \ref{thm:2nd moment}}\label{sec:proof-main}

Recall that 
\[
M_g(p^\nu, \ell, \omega;2)
= M^D+M^{OD}+M^{OOD}
+O_{\epsilon, v}\bigg( p^{\epsilon}{ \sum_{\substack{ab=d,  \\ de=\ell}} \frac{(ae)^{\theta}b^{\epsilon}}{\sqrt{ad}}} 
\sum_{M, N \ll p^{\nu(1+\epsilon)}}C^{\epsilon}\Big(\frac{\sqrt{aeMN}}{C}\Big)^{k-1-2\theta} \bigg).
\]
With the choice of $C$ given in \eqref{eqn:choiceC}, we get
\begin{equation}\label{eqn:M-total-error}
M_g(p^\nu, \ell, \omega;2)
= M^D+M^{OD}+M^{OOD}
+O\Big(\ell^{\left(\frac54-\theta\right)\frac{k+k'-2}{k-2\theta}}
p^{-\nu\left(\frac{k+k'-2}{4(k-2\theta)}-O(\epsilon)\right)}\Big).
\end{equation}
Applying Proposition \ref{prop:MD plus MOD} and plugging the off-off-diagonal bound \eqref{eqn:bound-MOOD} into  \eqref{eqn:M-total-error} end the proof.


\section{Proof of Corollary \ref{cor:thm 2nd moment}}\label{sec:cor}

We now derive Corollary \ref{cor:thm 2nd moment} from Theorem \ref{thm:2nd moment}.  We prove the corollary only in the case $\ell=1$, since the general case is entirely analogous but notationally much heavier. The essential cancellation phenomenon already appears in the untwisted case.

By Proposition \ref{prop:MD plus MOD} and its proof in Section \ref{sec:main term OD}, we have
    \begin{align*}
     &M_g^{D}(p^\nu, 1,\omega;2) +M_g^{OD}(p^\nu, 1, \omega;2)
     \\
    &=\, \mathscr{M}^{OD}_{1;2}
    +   \varepsilon_1(f\otimes g ) \mathscr{M}^{OD}_{2;2}
    +   \varepsilon_2(f\otimes g ) \mathscr{M}^{OD}_{3;2}
    +   \varepsilon_1(f\otimes g )   \varepsilon_2(f\otimes g ) 
    \mathscr{M}^{OD}_{4;2}
    \\
    &=  \Big(1-\frac{1}{p}\Big)
    \frac{(4\pi^2)^{2\omega_1}\Gamma_g(\frac12-\omega_1)}{\Gamma_g(\frac12+\omega_1) }
    p^{-\nu(\omega_1+\omega_2)}  p^{-\nu\omega_1+\nu\omega_2}  
    \\
    &\qquad \times 
    \zeta_{p^\nu}(1-2\omega_1)    
    \zeta_{p^\nu}(1+2\omega_2)
      \Upsilon(1-\omega_1+\omega_2, p, 1) \frac{L(1-\omega_1+\omega_2, g\otimes g)}{\zeta(2-2\omega_1+2\omega_2)}
    \\
    &+\,\,     \Big(1-\frac{1}{p}\Big)
    p^{-\nu(\omega_1+\omega_2)}  p^{\nu\omega_1+\nu\omega_2}  
    \\
    &\qquad \times 
    \zeta_{p^\nu}(1+2\omega_1)
    \zeta_{p^\nu}(1+2\omega_2)
            \Upsilon(1+\omega_1+\omega_2, p, 1) \frac{L(1+\omega_1+\omega_2, g\otimes g)}{\zeta(2+2\omega_1+2\omega_2)}
    \\   
    &+\,\,   \Big(1-\frac{1}{p}\Big)   \frac{(4\pi^2)^{2\omega_2}\Gamma_g(\frac12-\omega_2)}{\Gamma_g(\frac12+\omega_2) }    
    \frac{(4\pi^2)^{2\omega_1}\Gamma_g(\frac12-\omega_1)}{\Gamma_g(\frac12+\omega_1)}
    p^{-\nu(\omega_1+\omega_2)}  p^{-\nu\omega_1-\nu\omega_2} 
    \\ 
    &\qquad \times 
    \zeta_{p^\nu}(1-2\omega_1)
    \zeta_{p^\nu}(1-2\omega_2) 
            \Upsilon(1-\omega_1-\omega_2, p, 1) \frac{L(1-\omega_1-\omega_2, g\otimes g)}{\zeta(2-2\omega_1-2\omega_2)}
    \\
    &+\,\,  \Big(1-\frac{1}{p}\Big)  \frac{(4\pi^2)^{2\omega_2}\Gamma_g(\frac12-\omega_2)}{\Gamma_g(\frac12+\omega_2) }   
    p^{-\nu(\omega_1+\omega_2)}  p^{\nu\omega_1-\nu\omega_2} 
    \\ 
    &\qquad \times 
    \zeta_{p^\nu}(1+2\omega_1)
    \zeta_{p^\nu}(1-2\omega_2)
      \Upsilon(1+\omega_1-\omega_2, p, 1) \frac{L(1+\omega_1-\omega_2, g\otimes g)}{\zeta(2+2\omega_1-2\omega_2)} .
    \end{align*}
Next, we combine the terms according to their poles at $\frac{1}{\omega_1-\omega_2}$ or $\frac{1}{\omega_1+\omega_2}$. We define
    \begin{align*}
    \mathscr{F}_1(\omega_1, \omega_2)
    &= \mathscr{M}^{OD}_{1;2}
     +   \varepsilon_1(f\otimes g )   \varepsilon_2(f\otimes g ) 
    \mathscr{M}^{OD}_{4;2}
    \\
    \mathscr{F}_2(\omega_1, \omega_2) 
    &=  \varepsilon_1(f\otimes g ) \mathscr{M}^{OD}_{2;2}
    +   \varepsilon_2(f\otimes g ) \mathscr{M}^{OD}_{3;2} ,
    \end{align*}
so that     
    \[
    M^D+M^{OD}
    =\mathscr{F}_1(\omega_1, \omega_2)+\mathscr{F}_2(\omega_1, \omega_2) .
    \]
Firstly, we consider $\mathscr{F}_1$. Set 
    \begin{align*}
    f_1(\omega):=\Big(\frac{4\pi^2}{p^{\nu}}\Big)^{2\omega},
    \qquad & f_2(\omega):=\frac{\Gamma_g\left(\frac12-\omega\right)}{\Gamma_g\left(\frac12+\omega\right)},
    \\
    f_{3}(\omega):=\omega \, \zeta_{p^\nu}\left(1+\omega\right),
    \qquad  &
     f_4(\omega):=\omega L(1+\omega, g\otimes g) ,
    \\
    f_5(\omega):=\Upsilon(1+\omega, p,1), 
    \; \text{and}& \quad f_6(\omega):=\frac{1}{\zeta(2+2\omega)}.
    \end{align*}
Using this notation, we can write
    \begin{align*}
    \left(1-\frac{1}{p}\right)^{-1}
    \mathscr{F}_1(\omega_1,\omega_2)
    &=\frac{\mathcal{G}_1(\omega_1)\mathcal{G}_2(\omega_2)}{-4\omega_1\omega_2(\omega_2-\omega_1)}\mathcal{H}(\omega_2-\omega_1)+\frac{\mathcal{G}_1(\omega_2)\mathcal{G}_2(\omega_1)}{-4\omega_1\omega_2(\omega_1-\omega_2)}\mathcal{H}(\omega_1-\omega_2),
    \end{align*}
with 
    \[
    \mathcal{G}_1(\omega)=f_1(\omega)f_2(\omega)f_3(-2\omega), 
    \quad \mathcal{G}_2(\omega)=f_3(2\omega) \quad
    \text{and}\quad \mathcal{H}(\omega)=f_4(\omega)f_5(\omega)f_6(\omega).
    \]
We expand the entire function $\mathcal{H}(\omega)$ into a Taylor series near 0 as follows.
    \[
    \mathcal{H}(\omega)=h_0+h_{1}\omega+\sum_{j\geq2}h_j\omega^{j}.
    \]
Then by direct substitution, we find that
    \begin{equation}\label{eq:A1 before lemma}
    \begin{split}
    &\left(1-\frac{1}{p}\right)^{-1}\mathscr{F}_1(\omega_1,\omega_2)
    \\
    &=\frac{\mathcal{G}_1(\omega_1)\mathcal{G}_2(\omega_2)}{4\omega_1\omega_2(\omega_1-\omega_2)}
    \bigg(h_0+h_1(\omega_2-\omega_1)+\sum_{j\geq2}h_j(\omega_2-\omega_1)^j\bigg)
    \\& \hspace{1cm} -\frac{\mathcal{G}_1(\omega_2)\mathcal{G}_2(\omega_1)}{4\omega_1\omega_2(\omega_1-\omega_2)} \bigg(h_0+h_1(\omega_1-\omega_2)+\sum_{j\geq2}h_j(\omega_1-\omega_2)^j\bigg)\\
    &=\frac{h_0}{4\omega_1\omega_2}\, \frac{\mathcal{G}_1(\omega_1)\mathcal{G}_2(\omega_2)-\mathcal{G}_1(\omega_2)
    \mathcal{G}_2(\omega_1)}{\omega_1-\omega_2}
    -\frac{h_1}{4\omega_1\omega_2}\big(\mathcal{G}_1(\omega_1)\mathcal{G}_2(\omega_2)+\mathcal{G}_1(\omega_2)\mathcal{G}_2(\omega_1)\big)\\
    &\quad -\frac{\mathcal{G}_1(\omega_1)\mathcal{G}_2(\omega_2)}{4\omega_1\omega_2}\sum_{j\geq2}h_j(\omega_2-\omega_1)^{j-1}-\frac{\mathcal{G}_1(\omega_2)\mathcal{G}_2(\omega_1)}{4\omega_1\omega_2}\sum_{j\geq2}h_j(\omega_1-\omega_2)^{j-1} .
    \end{split}
    \end{equation}
At this stage, we can apply the following lemma, which is rather straightforward to prove. 


\begin{lem}
Let $F_1$ and $F_2$ be entire functions. Then 
        \[
        \lim_{z_2\to z_1} \frac{F_1(z_1)F_2(z_2)-F_1(z_2) F_2(z_1)}{z_1-z_2}
         =   F_1'(z_1) F_2(z_1) -F_1(z_1) F_2'(z_1).
         \]
 \end{lem}

By using this lemma on the last line of \eqref{eq:A1 before lemma}, we obtain
    \begin{align*}
    \lim_{\omega_2\to\omega_1}\left(1-\frac{1}{p}\right)^{-1}\mathscr{F}_1(\omega_1,\omega_2)
    &=\frac{1}{\omega_1^2}\mathscr{G}_1(\omega_1),
    \end{align*}
where 
    \begin{align*}
    &\mathscr{G}_1(\omega_1)
    =\frac{h_0}{4}\Big\{\mathcal{G}'_1(\omega_1)\mathcal{G}_2(\omega_1)-\mathcal{G}_1(\omega_1)\mathcal{G}'_2(\omega_1)\Big\}
    -\frac{h_1}{2}\mathcal{G}_1(\omega_1)\mathcal{G}_2(\omega_1)
    \\&=\frac{h_0}{4}f_3(2\omega_1)
    \Big\{f'_1(\omega_1)f_2(\omega_1)f_3(-2\omega_1)+f_1(\omega_1)f'_2(\omega_1)f_{3}(-2\omega_1)-2f_1(\omega_1)f_2(\omega_1)f_{3}'(-2\omega_1)\Big\}
    \\&
    \hspace{1em}
    -\frac{h_0}{2}f_1(\omega_1)f_2(\omega_1)f_3(-2\omega_1)f'_3(2\omega_1)-\frac{h_1}{2}f_1(\omega_1)f_2(\omega_1)f_3(-2\omega_1)f_{3}(2\omega_1).
    \end{align*}
Moreover, from the definitions of $\mathcal{G}_1$ and $\mathcal{G}_2$, we find that
    \begin{align*}
    \lim_{\omega_2\to\omega_1}\left(1-\frac{1}{p}\right)^{-1}\mathscr{F}_2(\omega_1,\omega_2)
    =\frac{1}{\omega_1^3}\mathscr{G}_2(\omega_1),
    \end{align*}
where 
    \begin{align*}
    \mathscr{G}_2(\omega_1)
    &=-\frac{1}{8}f_1(2\omega_1)f^2_{2}(\omega_1)f^2_3(-2\omega_1)f_{4}(-2\omega_1)f_{5}(-2\omega_1)f_6(-2\omega_1)
    \\
    &\hspace{1em}+\frac{1}{8}f^2_{3}(2\omega_1)f_{4}(2\omega_1)f_5(2\omega_1)f_6(2\omega_1)\\
    &=-\frac{1}{8}f_1(2\omega_1)f^2_{2}(\omega_1)f^2_3(-2\omega_1)\bigg(h_0-2h_1 \omega_1+4h_2\omega_1^2-8h_3\omega_1^3+\sum_{j\geq4}h_j(-2\omega_1)^{j}\bigg)\\
    &\hspace{1em}+\frac{1}{8}f^2_{3}(2\omega_1)\Big(h_0+2h_1 \omega+4h_2\omega_1^2+8h_3\omega_1^3+\sum_{j\geq4}h_j(2\omega_1)^{j}\Big).
    \end{align*}
Observe that both $\mathscr{G}_1$ and $\mathscr{G}_2$ are entire functions in $\omega_1$. Using Maple, we verify that 
    \[
    \mathscr{G}_2(0)
    =\mathscr{G}_1(0)+\mathscr{G}_2'(0)
    =\mathscr{G}'_1(0)+\frac12\mathscr{G}_2''(0)=0.
    \] 
Hence,
    \[
    \lim_{\omega_1\to 0}\frac{1}{\omega_1^2}\mathscr{G}_1(\omega_1)
    +\frac{1}{\omega_1^3}\mathscr{G}_2(\omega_1)
    =\frac12\mathscr{G}_1''(0) +\frac16\mathscr{G}_2'''(0).
    \]
One can verify using Maple that $\frac12\mathscr{G}_1''(0)+\frac16\mathscr{G}_2'''(0)$ is a cubic polynomial in $\nu$ with leading coefficient
    \[
   \frac{2}{\pi^2} \Upsilon(1,p,1)\frac{(p-1)^2 \log^3{p}}{p^2}\mathrm{Res}_{s=0}L(1+s, g\otimes g).
    \]
The remaining coefficients were also computed in Maple, but we omit their explicit expressions here due to their length.


\end{document}